\documentclass[9pt,twoside]{amsart}
\usepackage{amsxtra}
\usepackage{color}
\usepackage{amsmath,amsthm,amssymb,bm}
\usepackage{mathrsfs,mathtools}
\usepackage{stmaryrd}
\usepackage{hyperref}
\usepackage{enumerate}
\usepackage{xcolor}
\usepackage{pifont}
\usepackage{adjustbox}
\usepackage{tikz}
\usepackage{comment}
\usepackage{tikz-cd}

\newtheorem{theorem}{Theorem}[section]
\newtheorem*{MainTheorem}{Main Theorem}
\newtheorem{corollary}[theorem]{Corollary}
\newtheorem{proposition}[theorem]{Proposition}
\newtheorem{lemma}[theorem]{Lemma}
\newtheorem*{theorem*}{Theorem}

\theoremstyle{definition}

\newtheorem{example}[theorem]{Example}

\theoremstyle{remark}
\newtheorem{remark}[theorem]{Remark}

\newcommand{\R}{{\mathbb R}}

\newcommand{\C}{{\mathbb C}}

\newcommand{\beq}{\begin{equation}}
\newcommand{\eeq}{\end{equation}}
\renewcommand{\a}{\alpha}

\newcommand{\g}{\gamma}

\renewcommand{\l}{\lambda}

\renewcommand{\o}{\omega}

\newcommand{\cH}{\mathcal{H}}

\newcommand{\SU}{{\mathrm{SU}}}

\newcommand{\n}{\nabla}

\newcommand{\W}{\wedge}

\DeclareMathOperator\tr{tr\;}

\DeclareMathOperator\Ad{Ad}
\DeclareMathOperator\ad{ad}

\newcommand{\Ric}{{\rm Ric}}

\newcommand{\so}{\mathfrak{so}}
\newcommand{\su}{\mathfrak{su}}

\newcommand{\dvol}{{\mathrm{vol}}}
\renewcommand{\tr}{\mathrm{tr}}

\numberwithin{equation}{section}

\title[New inhomogeneous BHE 3-folds with large symmetry group]{New inhomogeneous Bismut Hermitian-Einstein threefolds\\ with large symmetry group}

\author{Fabio Podest\`a and Alberto Raffero}
\address{Dipartimento di Matematica e Informatica ``U.~Dini'' \\ Universit\`a degli Studi di Firenze\\ Viale Morgagni 67/a\\ 50134 Firenze\\ Italy}
\email{fabio.podesta@unifi.it}
\address{Dipartimento di Matematica ``G.~Peano'' \\ Universit\`a degli Studi di Torino\\
Via Carlo Alberto 10\\
10123 Torino\\ Italy}
\email{alberto.raffero@unito.it}

\subjclass[2020]{53C55, 53C29, 53C25}
\keywords{Bismut Hermitian-Einstein metrics; pluriclosed metrics; cohomogeneity-one metrics.}

\begin{document}

\begin{abstract}
We study compact non-Kähler Bismut Hermitian–Einstein threefolds endowed with an automorphism group acting with cohomogeneity one. 
We prove that the Lie algebra of this group is isomorphic to $\mathbb R^3\oplus\su(2)$ or $\mathbb R^2\oplus\su(2)$, 
provided that the Bismut connection is not flat. 
In the first case, we explicitly describe all these structures up to homothety and biholomorphic isometry. 
We also construct explicit families of complete non-compact examples. The resulting non-Bismut-flat threefolds have non-constant scalar curvature 
and restricted Bismut holonomy $T^1\subset \mathrm{SU}(3)$, and include examples on $S^3\times S^3$ 
and complete examples on $S^3\times S^1\times\mathbb R^2$. 
Finally, we construct a smooth deformation of a Bismut-flat Samelson structure on $S^3\times S^3$ through non-Bismut-flat Bismut Hermitian–Einstein structures 
whose complex structures are not compatible with any Bismut-flat metric.
\end{abstract}

\maketitle

\begingroup
\edef\savedtheorem{\arabic{theorem}}
\setcounter{theorem}{0}
\renewcommand{\thetheorem}{\Alph{theorem}}

\section{Introduction}
Let $(M,g,J)$ be a Hermitian manifold with fundamental form $\omega=g(J\cdot,\cdot)$. 
The Bismut connection $\nabla^B$ is the unique connection preserving $g$ and $J$ whose torsion is totally skew-symmetric. 
According to the literature, we call $(M,g,J)$ {\em Bismut Hermitian-Einstein} (abbreviated {\em BHE}) if the metric is pluriclosed and the Bismut Ricci form of $\nabla^B$ is zero:
\beq\label{eq:BHE}
dd^c\omega=0,\qquad \rho^B=0.
\eeq

The study of these equations has attracted considerable attention in recent years as part of the search for canonical metrics on non-K\"ahler complex manifolds, see, e.g., \cite{ABLS,ALL,ALLR,Ba,Ba2,BFG,FG,GFJS,PR,Str}. 
They extend the K\"ahler Ricci-flat condition to Hermitian geometry and admit an interpretation as a Hermitian-Einstein equation on an associated 
holomorphic Courant algebroid \cite{GFJS}.
BHE metrics are also stationary solutions of the pluriclosed flow \cite{ST1,GFJS} and, in the compact case, determine steady
gradient generalized Ricci solitons \cite{GFJS,GFS,Str,ST2}.
More precisely, setting $\cH\coloneqq-d^c\omega$, there exists a smooth function $f$ such that
\begin{equation}\label{GRS}
\begin{cases}
\mathrm{Ric}_g-\dfrac14 \cH^2+\nabla^2f=0, \vspace{0.1cm} \\ 
\delta_g\cH+\imath_{\nabla f}\cH=0,
\end{cases}
\end{equation}
where $\cH^2(X,Y) = \sum_{i,j}\cH(X,e_i,e_j)\cH(Y,e_i,e_j)$, for a local $g$-orthonormal frame $\{e_i\}$. 
The soliton potential can be chosen so that $e^{-f/2}$ is the positive $L^2$-normalized eigenfunction corresponding to the
lowest eigenvalue of a certain Schr\"odinger operator, see \cite[Prop.~2.6]{GFJS}. 

\smallskip

Bismut-flat pluriclosed metrics provide a basic class of solutions to the BHE equations \eqref{eq:BHE}. 
In the compact case, their geometry is described by the classification of Wang-Yang-Zheng \cite{WYZ}, who proved that their universal cover is a Samelson space, 
namely a product $G\times\R^k$, where $G$ is a compact, simply connected, semisimple Lie group endowed with a bi-invariant metric and a compatible 
left-invariant complex structure.

\smallskip

In the non-K\"ahler case, the Lee form $\vartheta \coloneqq J\delta_g\omega$ and the soliton potential determine two commuting holomorphic Killing vector fields
\[
V\coloneqq \frac12\left(\vartheta^\sharp-\nabla f\right),\qquad JV,
\]
which are $\nabla^B$-parallel and have constant nonzero length, see \cite[Prop.~2.6]{ABLS} and \cite{GFJS}. 
These vector fields provide a natural starting point for the study of compact BHE manifolds.

\smallskip

In complex dimension three, Apostolov, Barbaro, Lee and Streets \cite{ABLS} established a canonical reduction of the BHE equations \eqref{eq:BHE} 
to the geometry transverse to the foliation $\mathcal{V}$ generated by $V$ and $JV$. 
The transverse Hermitian structure is conformally K\"ahler, and the equations can be expressed in terms of a scalar equation 
for the corresponding K\"ahler metric. 
Apostolov, Lahdili and Lee \cite{ALL} developed this approach in the quasi-regular case and constructed families of compact BHE structures on
$S^3\times S^3$ and $S^1\times S^2\times S^3$ which are not locally isometric to a Samelson geometry. 
Their examples arise from explicit solutions on orthotoric K\"ahler orbifold surfaces.
Further compact examples, with different underlying topologies, were obtained by Apostolov, Lahdili, Lee and Rochon \cite{ALLR} 
through a gluing construction for the transverse equation.

\smallskip

In this work, we use techniques from the theory of cohomogeneity one Lie group actions to construct new examples of BHE metrics 
in complex dimension three.
This approach has proved to be fruitful in a variety of settings, leading to the construction of new compact Einstein metrics 
and other special geometric structures, see e.g. \cite{Bo,FH,SR}. 
In our previous work \cite{PR}, we studied $\mathrm{SO}(4)$-invariant BHE metrics on the affine quadric $Q_3\cong TS^3$. 
We established existence and completeness for a one-parameter family of invariant BHE metrics including Stenzel's K\"ahler Ricci-flat metric, 
showing in particular that the Bismut connection of the non-K\"ahler members of the family has full holonomy $\SU(3)$.  
This result was independently obtained by Foscolo and Garcia-Fern\'andez \cite{FG}. 
Their work also treats existence and nonexistence questions on other spaces associated with the conifold.

\smallskip

In the first part of this paper, we investigate compact non-K\"ahler BHE threefolds admitting a cohomogeneity one action by holomorphic isometries.

Our first observation is that the vector fields $V$ and $JV$ are preserved by every holomorphic isometry. 
Consequently, they correspond to elements of the center of the Lie algebra of the connected full group $G$ of holomorphic isometries, 
showing that $\dim(Z(G))\geq 2$. Combining this fact with the restrictions imposed by a cohomogeneity one action gives the following result (see Proposition \ref{T5}).

\begin{proposition}\label{prop:intro-groups}
Let $(M,g,J)$ be a compact connected non-K\"ahler BHE manifold of complex dimension three, and let $G$ be the connected full group
of holomorphic isometries. Suppose that $G$ acts with cohomogeneity one and that the Bismut connection is not flat. Then
\[
\mathfrak g\cong\mathbb R^3\oplus\mathfrak{su}(2)
\qquad\text{or}\qquad
\mathfrak g\cong\mathbb R^2\oplus\mathfrak{su}(2).
\]
\end{proposition}

We focus here on the geometry arising from the first possibility, leaving the study of the second one for future work.  

After passing to a finite cover of the acting group, we work with an almost effective action of $G = T^3\times\mathrm{SU}(2)$. 
First, we show that the principal isotropy algebra $\mathfrak{h}$ is generated by a vector $Z+A\in\mathfrak{g}$, 
for some non-zero $Z\in\mathfrak{z}$ and non-zero $A\in\mathfrak{su}(2)$, where $\mathfrak{z}$ is the center of $\mathfrak{g}$ (see Lemma \ref{lem:Zneq0}). 
Then, we focus on the singular isotropy algebras, obtaining the following result 
(see Proposition \ref{prop:singiso} and Proposition \ref{prop:noTypeI}).
\begin{proposition}
Let $(M,g,J)$ be a compact connected non-K\"ahler BHE threefold admitting an almost effective cohomogeneity one action of $G=T^3\times\mathrm{SU}(2)$ 
by holomorphic isometries. Suppose that the Bismut connection is not flat. Then, there exist two singular orbits which are complex and of real codimension two.
\end{proposition}
We note that the preliminary isotropy analysis in Proposition \ref{prop:singiso} yields two alternatives: 
singular orbits of real codimension four with isotropy algebra $\R Z\oplus\mathfrak{su}(2)$ (Type I), 
and singular orbits of real codimension two with isotropy algebra $\mathfrak{h}\oplus\R U$, for some $U\in\mathfrak{z}\smallsetminus\{0\}$ (Type II). 
In the compact non-Bismut-flat setting, the differential equations and the smoothness conditions exclude Type I (Proposition \ref{prop:noTypeI}). 
The argument does not establish the corresponding exclusion in the non-compact setting.

\smallskip

In Section \ref{sec:gJgamma} we describe the $G$-invariant Hermitian structures along a normal geodesic. 
We then impose the pluriclosed condition and the vanishing of the Bismut Ricci form, 
obtaining a system of ordinary differential equations for the coefficients of the metric,  see sections \ref{sec:pluriclosed} and \ref{sec:rhoB0}.

The extension across a singular orbit is an essential part of the construction. 
In Section \ref{sub:extend}, we determine the smoothness conditions for both the metric and its fundamental form at a singular orbit 
of codimension two (Type II). 
Then, in Section \ref{sec:phiconst} we show that the system of ordinary differential equations together with the smoothness conditions imposed 
by the extension across two singular orbits of Type II forces one of the metric coefficients to be constant (Proposition \ref{D}). 
This has a very strong consequence, as it allows us to explicitly integrate the system and analyze the 
solutions on their maximal definition intervals, see Section \ref{sec:Ansatz}. 
Although our main focus is on compact BHE threefolds, the analysis also yields families of complete non-compact examples.
The resulting global constructions can be summarized as follows.

\begin{MainTheorem}
Let $G = T^3\times \SU(2)$.
\begin{enumerate}[$(1)$]
\item Up to homothety and biholomorphic isometry, all compact connected non-K\"ahler BHE threefolds with non-flat Bismut connection 
admitting an almost effective cohomogeneity one action of $G$ by holomorphic isometries are given by the explicit families
\[
(M,g_{m,a,b},J_{m,a,b}),\qquad m\in\mathbb N_{>0},\quad a^2+b^2>1,\quad b\ne-1,
\]
with the admissible isotropy subgroups and slice representations specified in Theorem \ref{Tcompact}. 
Every such structure has non-constant Riemannian scalar curvature and restricted Bismut holonomy isomorphic to $T^1 \subset \SU(3)$. 
These families include structures on $S^3\times S^3$ whose complex structures admit no compatible Bismut-flat metric.
\item For $m\in\mathbb N_{>0}$ and $a,b\in\mathbb R$ satisfying
\[
0<a^2+b^2\le1,\qquad b>-1,
\]
there exist explicit complete non-compact non-K\"ahler BHE structures $(g_{m,a,b},J_{m,a,b})$ on associated bundles $M=G\times_K\mathbb C$, 
with a unique singular orbit of real codimension two. 
These structures have non-constant Riemannian scalar curvature and restricted Bismut holonomy isomorphic to $T^1\subset \SU(3)$.    
\end{enumerate} 

\end{MainTheorem}

The detailed statements of the results specify the isotropy subgroups and the integral conditions required
for the global constructions, see Theorem \ref{thm:noncptex} and Theorem \ref{Tcompact}. 
In the compact case, we provide explicit choices that satisfy these conditions in Example \ref{ex:BHEcpt}.  
We also discuss the topology of the resulting manifolds, showing in particular that they are finitely covered by $S^3\times S^3$ in the generic case $a\neq0$ 
(Proposition \ref{prop:topology}). 

\smallskip
In the recent article \cite{Ba2}, Barbaro proves that on a compact simply connected Lie group endowed with a Bismut-flat Hermitian structure 
$(g,J)$ there exist a smooth deformation $J_t$ of the complex structure $J=J_0$, with $t$ varying in an open deformation space 
$\mathcal{U}\subseteq \R^l$ containing $0$, and a family $g_t$ of $J_t$-compatible pluriclosed metrics depending smoothly on $t$. 
Here, we prove that Bismut-flat structures may also admit a smooth deformation through non-K\"ahler BHE structures 
that are not Bismut-flat, by showing the following result (see Theorem \ref{thm:BHEdeformation}). 

\begin{theorem} 
On the manifold $S^3\times S^3$ there exists a family of BHE structures $(g_\epsilon,J_\epsilon)$ depending smoothly on a real parameter 
$\epsilon \in [0,\delta)$, for some $\delta>0$, such that
\begin{enumerate}[$\bullet$]
\item for $\epsilon=0$, $(g_0,J_0)$ is Bismut-flat and $J_0$ is a Samelson complex structure;
\item for $0<\epsilon<\delta$, $(g_\epsilon,J_\epsilon)$ is BHE and not Bismut-flat. 
\end{enumerate}
\end{theorem}

This theorem improves \cite[Theorem 1.6]{ALL}, where an infinite countable family of quasi-regular non-K\"ahler BHE structures is constructed on 
$S^3\times S^3$. 

\smallskip

The construction also allows a direct analysis of the Riemannian scalar curvature and of the restricted holonomy group 
of the Bismut connection  $\nabla^B$. The former is done in Section \ref{sec:scal}. 
In particular, we compute the values of the scalar curvature at the singular orbits and its limit at infinity in the complete non-compact cases. 
Comparing these values proves that the metrics are not locally homogeneous (Corollary \ref{cor:scalcompl} and Corollary \ref{cor:scalcpt}). 
Thus, the explicit solutions yield families of both complete non-compact and compact BHE threefolds whose geometry differs from that 
of the homogeneous Bismut-flat models. 
 In Section \ref{sec:hol}, we show that the restricted Bismut holonomy $\mathrm{Hol}^o(\nabla^B)$ is isomorphic to $T^1\subset \SU(3)$ 
(Proposition \ref{prop:hol}).

\smallskip

In Section \ref{sec:regularity}, we focus on the transverse K\"ahler geometry related to the foliation $\mathcal{V}$ generated by the two Killing fields $V,JV$,
and we prove that these vectors cannot generate a compact torus acting freely on the threefolds 
we have constructed, even when the action is effective. In other words, these threefolds are not {\em regular} in the sense of \cite{ALL}.
We also point out that quasi-regular cases do exist (see Example \ref{ex:quasireg}), and in principle they could be characterized in terms of algebraic 
conditions.

\smallskip

Finally, in Section \ref{sec:comparison} we compare our compact examples with those constructed in \cite{ALL,ALLR}. 
In Proposition \ref{prop:comparison1}, we show that the explicit orthotoric BHE structures of \cite{ALL} 
admit no cohomogeneity one action by holomorphic isometries, and hence are not holomorphically homothetic to our examples. 
In Proposition \ref{prop:comparison2}, we prove that the compact threefolds in \cite[Theorem~3]{ALLR} are distinguished from ours by their second Betti numbers.

\smallskip

\noindent{\bf Notation.} 
We denote Lie groups with capital letters, e.g., $G,H,K$, and their Lie algebras with the corresponding gothic letters, 
e.g., $\mathfrak{g}, \mathfrak{h},\mathfrak{k}$.  
If $X$ is an element of the Lie algebra $\mathfrak{g}$, then $\hat X$ denotes the induced vector field on the manifold which is acted on by the Lie group $G$.

\setcounter{theorem}{\savedtheorem}
\endgroup

\section{Cohomogeneity one actions on compact BHE manifolds}\label{sec:1}
Let $(M,g,J)$ be a compact connected non-K\"ahler BHE manifold, and let $G$ be the connected full group of holomorphic isometries of $(M,g,J)$. 

It is known (see \cite{GFJS} and \cite[Prop.~2.6]{ABLS}) that $(M,g,J)$ carries two holomorphic Killing vector fields 
\[
V \coloneqq \frac12 \left(\vartheta^\sharp -\nabla f\right),\qquad JV,
\]
where $\vartheta\coloneqq J\delta_g\omega$ is the Lee form of $(g,J)$ and $f$ is the soliton potential appearing in \eqref{GRS}. 
Both $V$ and $JV$ are Bismut-parallel, hence of constant non-zero norm. Moreover, they commute, as $V$ is holomorphic. 
A further consequence is the following. 
\begin{lemma}
    The center of the connected full group $G$ of holomorphic isometries of $(M,g,J)$ has dimension at least two. 
\end{lemma}
\begin{proof}
    Every holomorphic isometry preserves $g$ and $J$, therefore it preserves the Lee form $\vartheta = J\delta_g\omega$,  
    the Schr\"odinger operator and its unique positive $L^2$-normalized ground state, thus $f$. 
    Consequently, it also preserves both $V$ and $JV$. 
    Then, the vectors in $\mathfrak{g}$ corresponding to $V$ and $JV$ belong to its center. 
    Since $V$ and $JV$ are non-zero and orthogonal, the claim follows. 
\end{proof}

As the case where $G$ acts transitively on $M$ is well understood, leaving only Bismut-flat cases, at least when $M$ has finite fundamental group \cite{Ba}, 
we investigate the case where $G$ acts by cohomogeneity one on $M.$ We start with the following preliminary result.
\begin{proposition}\label{T5} 
If $(M,g,J)$ is a compact connected non-K\"ahler BHE threefold that is not Bismut-flat, then the Lie algebra $\frak g$ is one of the following
$$\mathfrak g = \mathbb R^3\oplus \su(2),\quad \mathfrak g = \mathbb R^2 \oplus \su(2).$$ 
\end{proposition}
\begin{proof} 
First of all, we know that $\frak g$ is a compact Lie algebra and therefore splits as $\frak g = \frak z \oplus \frak g_s$, where $\frak z$ is the center 
and $\frak g_s$ is the semisimple part. 
If $p\in M$ is a regular point and $\mathcal O = G\cdot p$ is its orbit, then $T_p\mathcal O$ contains the complex subspace spanned by $V_p,JV_p$, 
which is pointwise fixed by the principal isotropy $H\coloneqq G_p$ and moreover $T_p\mathcal O$ contains $J\xi_p$, where $\xi_p$ is a unit normal to the orbit. 
This implies that the fixed point set of $H$ in $T_pM$ is at least $4$-dimensional and therefore $\dim H\leq 1$, 
as the isotropy representation is faithful. This implies that $\dim\mathfrak g\in\{5,6\}$. 
As $\dim\frak z\geq 2$, we have the following possibilities 
$$
\frak g= \mathbb R^6,\ \mathfrak g= \mathbb R^5,\ \frak g= \mathbb R^3\oplus \su(2),\ \frak g=\mathbb R^2\oplus \su(2).
$$
The case $\frak g=\mathbb R^6$ is immediately ruled out, as the isotropy $H$ would be finite and therefore $G$ would act transitively. 
Then $M$ would be a torus with a left-invariant Hermitian structure, which is automatically K\"ahler and flat. 

We now have to deal with $\frak g = \mathbb R^5$. In this case $G$ is a torus with trivial principal isotropy and every principal orbit is a torus $T^5\cong G$. 

First of all, we observe that there are no exceptional orbits. Indeed, if $Q=G\cdot q$ is such an orbit, there exists $h\in G_q$ that fixes $Q$ pointwise 
and acts non trivially on the normal space. 
This implies that $Q$ is a connected component of the fixed point set of $h$, hence it should be a complex submanifold, while it is $5$-dimensional. 

We now claim that there are precisely two singular orbits. Indeed, if not, we have a principal bundle 
$$
T^5\to M\to S^1
$$
that is trivial and $M$ is diffeomorphic to $T^6$, so that $b_1(M)=6$. 
Recall that the pair $(g,\cH=-d^c\omega)$ is a steady gradient generalized Ricci soliton, namely it satisfies \eqref{GRS} for 
a certain positive normalized potential $f\in C^\infty(M)$. 
Then, the Bakry-Emery tensor $\Ric_f= \operatorname{Ric}^g +\n^2f$ is non-negative and therefore we can use the analogue of Bochner's formulas 
for the twisted Laplacian applied to $1$-forms. 
More precisely, we consider the weighted volume form $e^{-f}\dvol_g$ and the adjoint $\delta_f$ given by 
$$
\delta_f =\delta +\imath_{\n f}
$$
with corresponding $f$-twisted Laplacian 
$$
\Delta_f = d\delta_f + \delta_f d = \Delta + \mathcal L_{\n f}.
$$ 
We have six independent $1$-forms which are $\Delta_f$-harmonic. If $\a$ is such a form, the adapted Bochner formula yields (see e.g.~\cite[(2.12)]{L}, 
with $\phi=e^{-f}$)
$$
\frac 12 \Delta_f(|\a|^2) = (\Delta_f\a,\a) -|\n^{LC}\a|^2 - (\operatorname{Ric}^g+\n^2f)(\a^\sharp,\a^\sharp) = -|\n^{LC}\a|^2 
                            - \frac 12 |\imath_{\a^\sharp}\cH|^2,
$$
 where $|\cdot|$ denotes the exterior algebra norm. 
Integration over $M$ using the twisted volum form $e^{-f}\dvol_g$ yields $\n^{LC}\a=0$ and $\imath_{\a^\sharp}\cH=0$. 
As we have six independent parallel $1$-forms, we see that $\cH=0$ and $\n^{LC}$, $\n^B$ are flat, contradicting the hypothesis. 
Therefore, there exist two singular orbits $\mathcal O_1,\mathcal O_2$. 
As the principal isotropy is finite, the codimension of the orbits is two, $\mathcal O_i\cong T^4$, for $i=1,2$, and the singular isotropies are $K_1,K_2\cong S^1$.

We now investigate the topology of $M$, which is the union of two open disk bundles $D_1,D_2$ over the singular orbits. 
The intersection $D_1\cap D_2$ contains a principal orbit $\mathcal O \cong T^5$ as a deformation retract and the homomorphisms 
$q_i:\pi_1(\mathcal O)\to\pi_1(\mathcal O_i)$, $i=1,2$, induced by the inclusions are surjective. Therefore we have 
$$\pi_1(M) \cong \frac{\mathbb Z^5}{\ker(q_1)+\ker(q_2)}.$$ 
As the ranks of these kernels are one, we have that the rank of $\pi_1(M)$ is at least $3$, hence 
$$b_1(M) \geq 3.$$
Again, using the same argument as above, we see that we have at least three linearly independent, $\n^{LC}$-parallel $1$-forms whose dual fields annihilate $\cH$.
If $X_1,X_2,X_3$ are such independent fields, then $\n^BX_i=0$ and therefore also $\n^BJX_i=0$. 
This means that $\{X_i,JX_i\}_{i=1,2,3}$ generate a $\n^B$-parallel complex subbundle $\mathcal P$ of $TM$ of rank at least $2$. 
Hence the Bismut curvature on this subbundle $\mathcal P$ vanishes and it is possibly non zero only on the orthogonal bundle $\mathcal P^\perp$.
As the complex trace of $R^B$ vanishes, the curvature vanishes also on $\mathcal P^\perp$, hence $M$ is Bismut-flat. \end{proof}

\begin{remark} 
Note that the Bismut-flat Calabi-Eckmann manifold $\SU(2)\times \SU(2)$ is BHE and admits $T^2\times \SU(2)$- and $T^3\times \SU(2)$-actions 
of cohomogeneity one. 
\end{remark}

\begin{remark}\label{codim}
    Note that for every $p\in M$, the vectors $V_p,JV_p\in T_p(G\cdot p)$, since $V$ and $JV$ are holomorphic Killing vector fields. 
    Since they are nowhere zero and orthogonal, every $G$-orbit has dimension at least $2$, namely codimension at most $4$. 
\end{remark}

\section{The case $\mathfrak g = \R^3 \oplus \mathfrak{su}(2)$}\label{sec:2}
We now focus in the case $\mathfrak g = \R^3 \oplus \mathfrak{su}(2)$, where we can suppose $G= T^3\times \SU(2)$ with an {\it almost} effective action.  
The principal isotropy algebra $\frak h$ is abelian of dimension one and $\frak h\cap \frak z = \{0\}$ as the action is supposed to be almost effective. 
This implies that the projection of $\frak h$ onto the $\su(2)$-factor is not trivial and we can write 
$$
\frak h=\mathbb R (Z+A), 
$$
for some $0\neq A\in \su(2)$ and $Z\in\mathfrak z$.
\begin{lemma}\label{lem:Zneq0}
The condition $\rho^B=0$ implies that $Z\neq 0$.
\end{lemma}
\begin{proof} 
Suppose $Z=0$. Then $\frak h\subset\su(2)$ and the $\SU(2)$-orbit $\mathcal W$ through the regular point $p$ would be a $2$-sphere $S^2$ 
(up to a finite quotient). 
The isotropy representation of $\frak h$ leaves the orthogonal complement of $T_p\mathcal W$ fixed and acts non-trivially on $T_p\mathcal W$. 
Therefore the orbit $\mathcal W$ is complex and the pullback of $TM$ via the immersion map $\imath :\mathcal S^2\to  M$ splits as the sum of $TS^2$ 
and a trivial complex bundle. 
Therefore $\int_{S^2}i^*c_1(TM) = \int_{S^2}c_1(TS^2) = 2$. On the other hand $\rho^B=0$ forces $c_1(TM)=0\in H^2(M,\mathbb R)$, a contradiction.
\end{proof}
Therefore we can write 
$$\frak g = \frak h \oplus \frak p_o \oplus \frak p_1,$$
where $\frak  p_i$ are $\ad(\frak h)$-modules, with $\frak p_o$ trivial of dimension $3$ and $\frak p_1\subset \su(2)$ non-trivial of dimension $2$. 
We note that $H$ may be not connected in general, but $H$ normalizes $H^o$ and therefore $H\subset T^3\times \exp(\mathbb RA)$, 
hence its action on $\frak g/\frak h$ coincides with the $H^o$-action. 
In particular if $h\in H$, there is $h'\in H^o$ with the same isotropy representation, hence $hh'^{-1}$ acts trivially on the principal orbit, 
hence trivially on the whole manifold because the principal isotropy acts trivially on the normal space to the orbit. 
This means that $H=H^oF$, where $F$ is the finite ineffective kernel of the $G$-action.

\smallskip

The next result provides useful information on the topology of $M$.
\begin{proposition}\label{prop:singiso} 
Let $(M,g,J)$ be a compact connected non-Kähler BHE threefold with non-flat Bismut connection admitting an almost effective cohomogeneity one action of 
$G=T^3\times \SU(2)$ by holomorphic isometries.  
Then, there exist two singular orbits, which are complex. If $\mathcal O$ is a singular orbit with singular isotropy $K\supset H$, then either
$$
\frak k = \mathbb R Z\oplus \su(2),\quad \mathrm{codim}\, \mathcal O = 4, \qquad ({\rm{Type\ I}})
$$
or 
$$
\frak k = \frak h \oplus \mathbb R U,\ U\in \frak z\smallsetminus \{0\},\quad  \mathrm{codim}\, \mathcal O = 2. \qquad({\rm{Type\ II}})
$$
Moreover, at least one singular orbit is of Type II.
\end{proposition}
\begin{proof} 
We use the same arguments as in the proof of Proposition \ref{T5}. 
We first prove that there are no exceptional orbits. In order to prove this, suppose we have an exceptional orbit $E$ with isotropy $K=G_q$, $q\in E$, containing $H$. 
Then $K$ and $H$ have the same dimension and therefore $K$ normalizes $H^o$.  We consider the normalizer $N_G(H^o)$ and prove that it is connected. 
Indeed, let $u\in N_G(H^o)$ and we write $u=\exp(z)\exp(v)$ with $z\in\frak z$ and $v\in\su(2)$. Then $h\coloneqq \exp(v)\in N_G(H^o)$ and therefore $\Ad(h)(A)\in \mathbb R A$. 
This means that $\Ad(h)A=\pm A$. If $\Ad(h)A=-A$, then $\Ad(h)(Z+A)=Z-A$, a contradiction as $Z\neq 0$. Hence $\exp(v)\in C_{\SU(2)}(A)=\exp(\mathbb R A)$ and $N_G(H^o)= Z(G)^oH^o$. 
Now, $N_G(H^o)$ acts on $\frak g/\frak h = \frak g/\frak k \cong T_qE$ and being connected every element acts with positive determinant. 
On the other hand, if $k\in K\setminus H$, it acts as $-\mathrm{Id}$ on the normal space of $T_qE$ and therefore $\det k_* <0$, 
contradicting the fact that every complex linear map has positive determinant. 

We now prove that there are two singular orbits. Suppose that this is not the case, then we obtain a fibration $G/H\to M\to S^1$. 
This fibration becomes trivial on a suitable finite covering $M'$ of $M$, which can be endowed with the pull-back Hermitian structure. 
Therefore we can suppose that $M\cong S^1\times G/H$ and prove that in this case the metric is Bismut-flat. We start with the following fact.
\begin{lemma}\label{lem:b1GH}
We have $b_1(G/H)=2$.
\end{lemma}
\begin{proof} If $H$ is connected, this follows immediately from the homotopy sequence. 
Indeed,  $\pi_1(G)=\mathbb Z^3$ and moreover the induced map $\pi_1(H)\to \pi_1(G)$ is injective as the projection of $H\cong S^1$ onto the center of $G$ 
is a non-trivial loop. 

If $H$ is not connected we consider the cover $\pi: G/H^o\to G/H$ that is a Galois cover under the right action of the finite group $H/H^o$. Now $H/H^o$ is a subgroup of $Q \coloneqq N_G(H^o)/H^o$ which is connected. 
The right action of $Q$ on $G/H^o$ preserves a Riemannian metric that therefore descends to a metric of $G/H$. 
In $G/H^o$ there are two independent harmonic $1$-forms, which are left fixed by the connected $Q$ and therefore descend to two independent harmonic forms. 
This shows that $b_1(G/H)\geq 2$ and by a general fact $b_1(G/H)\leq b_1(G/H^o)$. Our claim follows.    
\end{proof} 
Hence $b_1(M)=1+b_1(G/H) = 3$ and we can apply the same argument used in the proof of Prop.\ref{T5} to show Bismut flatness. 
Therefore there exist two singular orbits. 

Let $q$ be a singular point, $\mathcal O$ its $G$-orbit and $K=G_q$ its stabilizer, which has dimension at least $2$. 
We recall that $\mathrm{codim}\, \mathcal O \leq 4$ by Remark \ref{codim}. 
If $\nu_q:\frak k\to {\frak o}(T_q\mathcal O^\perp)$ is the slice representation at $q$, we have the following possibilities, 
according to $\nu$ being faithful or not. 

If $\nu$ is faithful, then $\mathrm{codim}\, \mathcal O = 3$ or $4$, as $\mathrm{codim}\, \mathcal O  =2$ would imply $\dim \frak k =1$. 
In the first case, $\frak k$ should be isomorphic to $\so(3)$, 
hence $\frak k$ should coincide with the semisimple part of $\frak g$ and would therefore not contain $\frak h$. 
Hence we are left with the case $\mathrm{codim}\, \mathcal O=4$ and therefore $\frak k \cong \frak u(2)$ 
by looking at the list of compact Lie groups acting transitively in $S^3$. 
This means that $\frak k = \mathbb R Z\oplus \su(2)$ and therefore the action of the subgroup $Q=\exp(\mathbb RZ)$ is trivial on $\mathcal O$ 
and non trivial via $\nu$. This implies that $\mathcal O$ is the fixed point set of $Q$, hence a complex submanifold.

When $\nu$ is not faithful, then $\ker \nu = \frak h$ and looking at the list of compact groups which are transitive on spheres with trivial isotropy, 
we see that $\mathrm{codim}\, \mathcal O = 2$ or $4$. 
In the last case, $\frak k \cong \mathbb R \oplus \frak \su(2)$ and the semisimple part of $\frak k$ is therefore equal to $\frak g_s = \su(2)$. 
As in this case $\frak h$ is an ideal of $\frak k$, it should commute with $\su(2)$, while $Z+A$ does not, a contradiction. 
We are then left with the first case, namely when $\mathrm{codim}\, \mathcal O = 2$ and $\frak k/\frak h \cong \mathbb R$. 
Then $\frak k$ is generated by $\frak h$ and a non-zero element $U\in\frak z$. 
Moreover, the group generated by $\exp(U)$ fixes $\mathcal O$ pointwise and acts non-trivially on the normal bundle. Hence, the orbit $\mathcal O$, 
being a connected component of the fixed point set of a holomorphic isometry, is complex.  

We now prove the last claim. We suppose there exist two singular orbits $\mathcal O_1,\mathcal O_2$ of Type I 
and we fix a normal geodesic $\g:[0,T]\to M$ with $\g(0)\in \mathcal O_1$, $\g(T)\in\mathcal O_2$ and $G_{\g(t)}=H$ for $t\in (0,T)$. 
Then the isotropy subalgebras at $\g(0),\g(T)$ coincide with $\frak k= \mathbb R Z \oplus \su(2)$. We select $X,Y\in \su(2)$ with 
$$[X,Y]=2A,\quad [Y,A]=2X,\quad  [A,X]=2Y.$$ 
The subspace generated by $\hat X,\hat Y$ along $\gamma$ is a non-trivial $\frak h$-submodule and therefore $g(\hat X,\hat Y)=0$. 
Therefore we can also suppose that $J\hat X|_{\gamma(t)} = \hat Y|_{\gamma(t)}$ for every $t\in(0,T)$. 
If we now denote by $\xi$ the unit normal field in the regular part of $M$, we compute along the curve $\gamma$
\begin{align*}0=dd^c\omega(\xi,\hat A,\hat X,\hat Y) &= \frac d{dt}d^c\o(\hat A,\hat X,\hat Y) + d^c\o(\xi,\widehat{[A,X]},\hat Y) 
                 + d^c\o(\xi,\widehat{[Y,A]},\hat X) + d^c\o(\xi,\widehat{[X,Y]},\hat A)\\
{}&= \frac d{dt}d^c\o(\hat A,\hat X,\hat Y) = -\frac d{dt}d\o(J\hat A,J\hat X,J\hat Y)\\ 
{}&= -\frac d{dt}d\o(J\hat A,\hat X,\hat Y).\end{align*}
We have already noticed that the space spanned by $V,JV$ is fixed under the isotropy representation and is complex, so that at a regular point $\g(t)$ we have 
\begin{align*}\hat A &= p J\xi + q V + r JV \\
J\hat A &= -p\xi -r V + q JV\end{align*}
for some smooth functions $p,q,r$. We now compute 
$$d\o(\xi,\hat X,\hat Y) = \frac d{dt}\o(\hat X,\hat Y) + \o(\xi,\widehat{[X,Y]}) = \frac d{dt}\|\hat Y\|^2 + 2p,$$
$$d\o(JV,\hat X,\hat Y) = -2\o(\hat A,JV) = -2g(\hat A,V) = -2q,$$
$$d\o(V,\hat X,\hat Y) = -2\o(\hat A,V) = -2g(J\hat A,V) = 2r.$$
Therefore, if we set $\phi(t)= \|\hat Y\|^2$, 
$$d\o(J\hat A,\hat X,\hat Y) = -p(\phi'+2p)-r(2r) + q(-2q) = -p\phi' -2 (p^2+r^2+q^2).$$
As we can suppose that $V,JV$ have both unit length, we have $p^2+q^2+r^2 = \|\hat A\|^2$. Therefore 
\beq\label{eq}p\phi' + 2 \|\hat A\|^2 = {\rm{const}}.\eeq
As $\su(2)\subset \frak k$, we have that $\phi(0)=\phi(T)=0$ and also $\hat A$ vanishes at the singular point. 
We remark that the function $p$ extends to a smooth function along the curve $\g$ and therefore the constant in \eqref{eq} is zero. 
As $A\not\in \frak h$, the function $\phi'$ never vanishes in the interval $(0,T)$, while the function $\phi$ vanishes at the endpoints $t=0,t=T$,
a contradiction. This proves our last claim. \end{proof}

We will see later in Proposition \ref{prop:noTypeI} that the configuration with singular orbits of different types cannot occur when 
the BHE threefold $(M,g,J)$ is compact and not Bismut-flat.

In the next section, we will investigate the general invariant Hermitian metric.

\subsection{The metric and the complex structure on the regular part}\label{sec:gJgamma}
We choose a basis $A,X,Y$ of $\mathfrak{su}(2)$ satisfying
\beq\label{su2}
[X,Y]=2A,\qquad [Y,A]=2X,\qquad [A,X]=2Y,
\eeq
and we denote by  $\alpha,x,y$ the corresponding Maurer--Cartan forms. Thus
$$
d\alpha=-2x\wedge y,\qquad
dx=-2y\wedge\alpha,\qquad
dy=-2\alpha\wedge x.
$$
Set
$$
\Omega \coloneqq x\wedge y.
$$
In particular, $d\Omega=0$.

We now choose a basis $Z_1,Z_2,Z_3$ of the center $\frak z$ so that 
$$
\hat Z_1 = V,\qquad \hat Z_2= JV,
$$
and write
$$
Z=aZ_1+bZ_2+cZ_3,
$$
for some $a,b,c\in \mathbb R$. 
If $\sigma_1,\sigma_2,\sigma_3$ are the closed central forms dual to this basis, define
$$
\theta_1=\sigma_1-a\alpha,\qquad
\theta_2=\sigma_2-b\alpha,\qquad
\theta_3=\sigma_3-c\alpha.
$$
These forms annihilate $\mathfrak h$, descend to invariant forms on the principal orbit and are the dual $1$-forms to the fields $\hat Z_i$, $i=1,2,3$. They satisfy
$$
d\theta_1=2a\, \Omega,\qquad
d\theta_2=2b\, \Omega,\qquad
d\theta_3=2c\, \Omega.
$$
Although $x$ and $y$ separately depend on the choice of local section,
$x^2+y^2$ and $\Omega$ descend to the principal orbit.

We consider a normal geodesic $\g:[0,T]\to M$ with arclength parameter $t$. 
Recall that $V$ and $JV$ have constant lengths, so that we can suppose they are unitary.  
With this normalization, the invariant metric $g$ on the regular part $M_{\mathrm{reg}} \cong (0,T)\times G/H$ of $M$
can be written as follows along $\g$
$$
g = dt^2 + \theta_1^2 + \theta_2^2 + 2p\, \theta_1\theta_3 + 2q\, \theta_2\theta_3 + r^2\theta_3^2 + \phi (x^2+y^2),
$$
for suitable functions
$$
p=p(t),\qquad q=q(t),\qquad r=r(t)>0,\qquad \phi=\phi(t)>0.
$$
If we now put 
$$
\eta_1\coloneqq \theta_1+p\,\theta_3  = V^\flat,\qquad
\eta_2 \coloneqq \theta_2+q\,\theta_3   = (JV)^\flat,
$$
then we can write the metric as 
\beq\label{metric}
g=  dt^2 + \eta_1^2 + \eta_2^2 + \ell^2\,\theta_3^2 + \phi (x^2+y^2),
\eeq
for some positive function $\ell=\ell(t)$ with $\ell^2 = r^2-p^2-q^2$. 
If we set 
$$
W \coloneqq \hat Z_3-pV-qJV = \hat Z_3 - g(\hat Z_3,V)V-g(\hat Z_3,JV)JV,
$$
then $\ell^2=g(W,W)$ and the almost complex structure is characterized by
$$
J\hat Z_1=\hat Z_2,\qquad
J\xi=\frac{1}{\ell}W,\qquad
JW=-\ell\xi,\qquad J\hat X=\hat Y,  
$$ 
along the normal geodesic and then extended on the regular part by $G$-invariance. We then have that 
\beq
\omega = \eta_1\wedge\eta_2 +\ell\,dt\wedge\theta_3+\phi\, \Omega.
\eeq
We now prove that this almost complex structure, defined on the regular part, is indeed integrable, i.e., for vector fields $U_1,U_2$ the Nijenhuis tensor
$$N_J(U_1,U_2) = [JU_1,JU_2]-[U_1,U_2] - J[JU_1,U_2]- J[U_1,JU_2] =0.$$
As $N_J(JU_1,U_2)=N_J(U_1,JU_2)=-JN_J(U_1,U_2)$, we are left with proving the following 
\begin{itemize}
\item[(i)] $N_J(\xi,\hat Z_1) = 0$
\item[(ii)] $N_J(\xi,\hat X) = 0$
\item[(iii)] $N_J(\hat Z_1,\hat X) = 0$
\end{itemize}
where $\xi$ is the $G$-invariant unit vector field orthogonal to the regular orbits and tangent to $\gamma$.

As for (i), we have that $J\hat Z_1 = \hat Z_2$ everywhere and both commute with $\xi$ and $J\xi$, giving our claim.

As for (ii), we have 
\begin{align*}N_J(\xi,\hat X) &= [J\xi,J\hat X]-[\xi,\hat X] - J[J\xi,\hat X]- J[\xi,J\hat X]\\
{}&= [J\xi,J\hat X] - J[\xi,J\hat X]. \end{align*}
Now we observe the following. The field $J\hat X$ coincides with $\hat Y$ along the geodesic $\gamma$ and therefore 
$$[\xi,J\hat X]|_{\g(t)} = [\xi,\hat Y]|_{\gamma(t)} = 0.$$
Moreover on the regular part the fields $\hat Z_i$ are $G$-invariant for $i=1,2,3$ and therefore we have 
$$
J\xi = \frac{1}{\ell}W,\quad W= \hat Z_3- g(\hat Z_3,V)V-g(\hat Z_3,JV)JV,
$$
where $\ell = \|W\|$. We recall that $[\hat Z_i,J\hat X]=0$ as $\mathcal L_{\hat Z_i}J=0$ for $i=1,2,3,$ and therefore  
$$[J\xi,J\hat X] = -J\hat X\left(\frac 1\ell\right)W + \frac 1\ell [J\hat X(g(\hat Z_3,V))V + J\hat X(g(\hat Z_3,JV))JV].$$
Now, along the normal geodesic $J\hat X = \hat Y$ and the functions $\ell, g(\hat Z_3,V),g(\hat Z_3,JV)$ are $G$-invariant, so that $[J\xi,J\hat X]|_{\g(t)}=0$ and our claim follows. 

As for (iii) we have 
$$N_J(\hat Z_1,\hat X) = [\hat Z_2,J\hat X] - [\hat Z_1,\hat X] - J[\hat Z_2,\hat X] - J[\hat Z_1,J\hat X] = 0$$
as $\mathcal L_{\hat Z_i}J=0$ and $\hat Z_i$ are central for $i=1,2$. Therefore, the integrability condition for $J$ does not impose further conditions.

\subsection{Torsion and pluriclosedness}\label{sec:pluriclosed}

Differentiating $\omega$ gives
$$d\omega=dt\wedge\theta_3\wedge(p'\eta_2-q'\eta_1)+
\Omega\wedge
\left(\mathcal D\,dt+2\mathcal A\eta_2-2\mathcal B\eta_1\right),$$
where we have set
$$\mathcal A \coloneqq a+cp,\qquad\mathcal B \coloneqq b+cq,\qquad \mathcal D \coloneqq \phi'-2c\ell.$$
Note that the metric is K\"ahler specifically when
$$\mathcal A=\mathcal B =\mathcal D = 0.$$
Since $d^c\omega = - d\o(J\cdot,J\cdot,J\cdot)$, it follows that
\beq\label{H}
\cH = - d^c\o = dt\wedge\theta_3\wedge(p'\eta_1+q'\eta_2) +\Omega\wedge\left(-\ell\mathcal D\,\theta_3 +2\mathcal A\eta_1+2\mathcal B\eta_2\right). 
\eeq
A direct differentiation yields
$$
d\cH=\left(4\mathcal A p'+4\mathcal B q'-\left(\ell(\phi'-2c\ell)\right)'\right)dt\wedge\theta^3\wedge\Omega.
$$
Therefore pluriclosedness is equivalent to
\[
\left(\ell(\phi'-2c\ell)\right)' = 4(a+cp)p'+4(b+cq)q'.
\]
By integration we get 
\begin{equation}\label{pluriclosed}
\ell\phi' = 2c(\ell^2+p^2+q^2)+4ap+4bq+k,
\end{equation}
where $k$ is constant.

\begin{remark}
The fields $V$ and $JV$ are Killing and their metric dual 1-forms are $\eta_1$ and $\eta_2$, respectively.
By \eqref{H} we have 
$$
\iota_V\cH=d\eta_1,\qquad  \iota_{JV}\cH=d\eta_2.
$$
For a Killing field $U$ with metric dual $1$-form $\eta$, we have 
$$
g(\n^B_{X_1}U,X_2) = g(\n^{LC}_{X_1}U,X_2) +\frac 12 \cH(X_1,U,X_2) = \frac12\left(d\eta(X_1,X_2)-\cH(U,X_1,X_2)\right).
$$
Therefore 
$$
\nabla^BV=0,\qquad \nabla^B(JV)=0,
$$
and there are no additional conditions imposed by the fact that both $V$ and $JV$ are Bismut parallel. 
\end{remark}

\subsection{The Lee form and the Ricci forms}\label{sec:rhoB0}
In real dimension six, the Lee form $\vartheta=J\delta_g\omega$ of an (almost) Hermitian structure is equivalently characterized by
$$
d(\o^2) = \vartheta\wedge \o^2.
$$
In our case 
$$
\o^2 = 2\ell\, dt\wedge \eta_1\wedge \eta_2\wedge \theta_3 +2\ell\phi\, dt\wedge \theta_3\wedge \Omega + 2\phi\, \eta_1\wedge\eta_2\wedge\Omega,
$$
and 
$$
d(\o^2) = 2\, dt\wedge\left[\mathcal D\, \eta_1\wedge\eta_2 + (2\ell\mathcal B+\phi q')\eta_1\wedge\theta_3
            -(\phi p'+2\ell\mathcal A)\eta_2\wedge\theta_3\right]\wedge \Omega.
$$
Therefore, the Lee form is given by 
\begin{equation}\label{Lee}
\vartheta=h\,dt-v\eta_1+u\eta_2,
\end{equation}
where we have set 
\beq\label{uvh}
u \coloneqq \frac{p'}{\ell}+\frac{2(a+cp)}{\phi},
\qquad
v \coloneqq \frac{q'}{\ell}+\frac{2(b+cq)}{\phi},
\qquad
h \coloneqq\frac{\phi'-2c\ell}{\phi}.\eeq
In particular,
\beq\label{JLee}
J\vartheta=\ell h\, \theta_3-u\eta_1-v\eta_2, 
\eeq
where we have used the convention $J\vartheta = - \vartheta\circ J$.

\medskip

We now compute the Chern--Ricci form using a local unitary frame of the canonical bundle. Choose a local section of $G\to G/H$, and denote
the pullbacks of $\alpha,x,y$ by $\alpha_s,x_s,y_s$. Set
$$
\zeta_1=\eta_1+i\eta_2,\qquad
\zeta_2=dt+i\ell\theta_3,\qquad
\zeta_3=x_s+iy_s.
$$
These are local $(1,0)$-forms, and 
$$\begin{aligned}
d\zeta_1
&=(p'+iq')\,dt\wedge\theta_3
  +2\left(a+ib+c(p+iq)\right)\Omega,\\
d\zeta_2
&=i\ell'\,dt\wedge\theta_3+2ic\ell\,\Omega,\\
d\zeta_3
&=-2i\alpha_s\wedge\zeta_3.
\end{aligned}$$
Consider the local $(3,0)$-form $\Psi\coloneqq \zeta_1\wedge\zeta_2\wedge\zeta_3$. We have 
$$d\Psi=\left(\frac{\ell'}{\ell}\,dt-2i\alpha_s\right)\wedge\Psi.$$
Since $\|\Psi\|^2=C/\phi$, for a positive constant $C$, the frame $\widehat\Psi=\sqrt{\phi/C}\,\Psi$ is unitary. 
Writing $F=\log(\ell\sqrt{\phi})$, we obtain
\[
d\widehat\Psi=(dF-2i\alpha_s)\wedge\widehat\Psi.
\]
The Chern connection $\n^K$ on the canonical bundle $K$ is characterized by its $(0,1)$-part and compatibility with the Hermitian metric. 
If $\Gamma$ is the connection form, we have 
$$
\nabla^K\widehat\Psi = \Gamma\otimes \widehat\Psi,
$$
and by definition we have 
$$
\bar\partial \widehat\Psi = \Gamma^{0,1}\otimes\widehat\Psi.
$$
Therefore 
$$
\Gamma^{0,1} = \bar\partial F - 2i\a_s^{0,1}.
$$
Now, as $\widehat\Psi$ has unit norm, we have 
$$
0 = d\|\widehat\Psi\|^2 = \Gamma+\overline{\Gamma}.
$$
Thus, 
$$
\Gamma^{1,0}=-\overline{\Gamma^{0,1}} = -\partial F - 2i\a_s^{1,0}.
$$
It follows that 
$$
\Gamma =\bar \partial F-\partial F-2i\alpha_s.
$$
As $\partial = \frac 12(d+id^c)$ and $\bar\partial = \frac 12(d-id^c)$, we have $\bar\partial F - \partial F = -id^cF$ and 
$$
\Gamma=-i(JdF+2\alpha_s).
$$
Using $d\Gamma=i\rho^C$, $d\alpha_s=-2\Omega$, and $Jdt=\ell\theta_3$, we conclude that
\beq\label{Chern}
\rho^C=4\Omega-d\ \left[ \left(\ell'+\frac{\ell\phi'}{2\phi}\right)\theta_3 \right].
\eeq
Note that the chosen frame and its connection form depend on the local section, while the curvature expression is independent of that choice.

Using the fact that the Bismut Ricci form $\rho^B$ can be expressed as 
$$
\rho^B = \rho^C + d(J\vartheta)
$$
and combining \eqref{JLee} and \eqref{Chern}, we obtain
\begin{equation}\label{rhoB}
\rho^B=-u'\,dt\wedge\theta^1-v'\,dt\wedge\theta^2 -w'\,dt\wedge\theta_3+\left(4-2(au+bv+cw)\right)\Omega,
\end{equation}
where $u,v$ are given in \eqref{uvh} and we have set 
\beq\label{w}
w \coloneqq \ell'-\frac{\ell\phi'}{2\phi}+\frac{2c\ell^2}{\phi}+pu+qv.
\eeq

Consequently,
\beq\label{equationrho}
\rho^B=0 \quad\Longleftrightarrow\quad u,v,w\text{ are constant},\qquad au+bv+cw=2.
\eeq

\begin{lemma}\label{lem:BRF}
When $\rho^B=0$ and the metric is not K\"ahler, we can arrange 
$$
u=0,\qquad v=-2,
$$
so that the remaining algebraic equation in \eqref{equationrho} is the following.
\beq\label{alg}
-2b+cw=2.
\eeq
\end{lemma}
\begin{proof} 
We know that $u,v$ are constant functions. Recall that, by \cite[Prop 2.6]{GFJS}, the vector field 
$$
\frac 12(\vartheta^\sharp - \n f),
$$
where $f$ is the normalized soliton potential in \eqref{GRS}, is non-zero, as the metric is supposed to be non-K\"ahler. 
Moreover, it is Killing, hence it is orthogonal to $\xi$, and the potential $f$ is $G$-invariant, hence its gradient is parallel to $\xi$. 
Using the expression of $\vartheta$ in \eqref{Lee}, we then have  
$$
0=g(\vartheta^\sharp-\nabla f,\xi)= h-f'.
$$  
Therefore, we have that the form 
$$
\vartheta-df = -v\eta_1 + u\eta_2\neq 0,
$$
hence $\mu  \coloneqq (u^2+v^2)^{1/2}\neq 0$. We can now rotate the vectors $V,JV$, considering instead 
$$
\tilde V = \frac 1\mu (-vV+uJV),\quad J\tilde V= \frac 1\mu (-uV-vJV),
$$
and correspondingly their duals
$$
\tilde \eta_1 = \frac 1\mu (-v\eta_1+u\eta_2),\quad 
\tilde \eta_2 = \frac 1\mu (-u\eta_1-v\eta_2),
$$
so that with this choice, the new Lee form becomes 
$$
\tilde \vartheta = h\, dt + \mu\, \tilde \eta_1.
$$
This means that $\tilde u=0$ and $\tilde v = -\mu$.

Now, we can perform a suitable scaling of the metric. Indeed, we can define $\tilde g=s^2g$ for some positive $s$, 
and use a new arclength parameter $\tilde t=st$. 
The Lee form is not affected by this homothety and we can consider $\tilde V = \frac 1s V$, that is again of unit norm with respect to $\tilde g$, 
and $\tilde\eta_1 = s\eta_1$. 
Then 
$$
\vartheta = df + \mu \eta_1 = df + \frac{\mu}s \tilde \eta_1 = df + 2\tilde \eta_1,
$$
if we set $s = \frac\mu 2$. Thus $\tilde{u}=0$ and $\tilde{v}=-2$.
\end{proof}

\subsection{Extendability conditions across a singular orbit of Type II}\label{sub:extend}
We suppose that at the singular point $\gamma(0)$ we have a singular isotropy subalgebra of Type II, namely $\frak k = \frak h \oplus \mathbb R U$, 
for some $U \in\mathfrak{z}\smallsetminus\{0\}$. We can then arrange it from the beginning to have $U = Z_3$. 
The singular isotropy $K=G_{\g(0)}$ is compact, and $C\coloneqq (K\cap Z(G)^o)^o$ is a closed connected subgroup 
with Lie algebra $\mathfrak{k}\cap \mathfrak{z} = \R\,Z_3$. 
Thus, $C$ is a circle. Rescaling $Z_3$ if necessary, we may assume that it is a primitive integral generator of $C$,
so that $\exp(tZ_3)=1$ if and only if $t\in2\pi\mathbb{Z}$. 

We have $K/H\cong S^1$ and $K^o$ still acts transitively on $S^1$. This implies that $K=K^oH$. 
We recall that $H = H^oF$, where $F$ is the finite ineffectivity kernel of the $G$-action. Hence, $K=K^oH^oF=K^oF$. 
If now $\nu:K\to {\rm{O}}((T_{\g(0)}G\cdot {\g(0)})^\perp)\cong{\rm{O}}(2)$ is the slice representation, then $\nu(F)=\mathrm{Id}$ and therefore 
$$\nu(K)=\nu(K^o),\qquad \nu(H)={\rm{Id}}.$$
This means that we can just consider the infinitesimal action $\nu_*(\frak k)$. In particular, there exists $m\in\mathbb N$, $m\geq 1$, with 
$$\nu(\exp(s Z_3)) = e^{ims}.$$
Note also that $m=1$ if $F$ is trivial. Indeed, if $m\neq 1$, then we would have a non-trivial element $k\in\ker \nu\cap\exp(\mathbb R Z_3)$, 
that would act trivially on the slice and also on the orbit as $Z_3$ is a central element, showing that $k\in F$.

Now, the smoothness criteria of \cite[Section 3.1, Lemma 3.8(a) and Lemma 3.9(a)]{VZ} with $\mathfrak{p} = \R\,Z_3$ 
and $\mathfrak{m} = \mathrm{span}(Z_1,Z_2,X,Y)$ imply that the metric extends smoothly across this orbit if and only if
$$
p(t)=t^2P(t^2),\qquad 
q(t)=t^2Q(t^2),\qquad
\ell(t)=tL(t^2),\qquad
\phi(t)=\Phi(t^2),
$$
where $P,Q,L,\Phi$ are smooth near zero and
$$
L(0)=m,\qquad \Phi(0)>0.
$$
Indeed, the condition involving the collapsing of the vector field $\hat Z_3$ gives
$$
g(\widehat Z_3,\widehat Z_3) =m^2t^2+t^4R(t^2),
$$
while the mixed scalar products $g(\widehat Z_3,V)=p$ and $g(\widehat Z_3,JV)=q$ are $t^2$ times smooth functions of $t^2$. 
Since $g(\widehat Z_3,\widehat Z_3)=\ell^2+p^2+q^2$, these conditions provide the form of the function $\ell$.

Therefore evaluating the pluriclosedness condition \eqref{pluriclosed} at $t=0$ we immediately get $k=0$. 
Moreover the function $w$ is constant and by definition (see \eqref{w}) we have that $w=m$. It then follows from \eqref{alg} that
$$
mc=2(b+1).
$$

We now prove that the same conditions are also sufficient to provide the smooth extension of the fundamental form 
$$
\omega=(\theta_1+p\theta_3)\wedge(\theta_2+q\theta_3) +\ell\,dt\wedge\theta_3+\phi\Omega.
$$
Indeed, we first prove that the forms $\theta_1,\theta_2,\Omega$ extend smoothly across the singular orbit.
The left-invariant forms on $G$ inducing $\theta_1,\theta_2,\Omega$ are horizontal for $\mathfrak k=\mathfrak h\oplus\mathbb RZ_3$ 
and invariant under the right action of $K$. Indeed, $K^o$ fixes the central directions and rotates $\operatorname{Span}(X,Y)$, preserving $x\wedge y$; 
moreover, $K=K^oF$, where the ineffective kernel $F$ is central. Therefore, these forms descend to the forms $\bar\theta_1,\bar\theta_2,\bar\Omega$ in $G/K$. 
We consider a $G$-stable neighborhood of $\g(0)$ given by $G\times_K \Upsilon$, where $\Upsilon$ is a $2$-dimensional $K$-module 
with the action given by the slice representation and let $\pi: G\times_K \Upsilon\to G/K$ denote the bundle projection. 
Under the identification of $G\times_K (\Upsilon\smallsetminus\{0\})$ with $(0,\varepsilon)\times G/H$, this projection is $(t,gH)\mapsto gK$. 
Consequently, the smooth forms $\pi^*\bar\theta_1$, $\pi^*\bar\theta_2$ and $\pi^*\bar\Omega$ agree with $\theta_1,\theta_2,\Omega$ on the regular part 
and provide their smooth extensions across the singular orbit.

As for the terms involving $\theta_3$, we consider a $G$-stable neighborhood of $\g(0)$ given by $G\times_K \Upsilon$, 
where is a $2$-dimensional $K$-module with the action given by the slice representation. 
In $\Upsilon$ we can fix Cartesian coordinates $t,z$ so that the curve $\gamma(t)$ can be represented as the curve $t\mapsto (t,0)$. 
At every point $v$ of $\Upsilon$ we also have the tangent vectors at $T_vM$ given by $\hat Z_1,\hat Z_2,\hat X,\hat Y$ which, 
together with  $\frac\partial{\partial t},\frac \partial{\partial z}$, form a frame along $\Upsilon$. 
The form $\theta_3$ is $G$-invariant and therefore is completely determined by its values along $V$ and $\theta_3$ can be seen as a $K$-equivariant map 
$$
\vartheta_3:\Upsilon\smallsetminus\{0\}\to \Upsilon^*\oplus \frak m^*
$$
where $\frak m = \operatorname{Span}(Z_1,Z_2,X,Y)$. As 
$$\hat Z_3|_{\gamma(t)}= \frac d{ds}|_{s=0}\exp(sZ_3)(t,0) = t(0,m),$$
we see that 
$$\hat Z_3|_{\gamma(t)} = tm\frac \partial{\partial z}$$
and 
$$\vartheta_3|_{\gamma(t)} = \frac 1{tm}dz.$$
This means that the map $\vartheta_3$ has no component along $\frak m^*$. Given $(t,z)\in \Upsilon$ we can find $h\in \exp(\mathbb R Z_3)$ with 
$\nu(h)(t,z)=\sqrt{t^2+z^2}(1,0)$, say $\nu(h) = \left(\begin{smallmatrix} \cos \phi& -\sin\phi\\ \sin\phi&\cos\phi \end{smallmatrix}\right)$, 
and therefore 
$$
\theta_3|_{(t,z)}\left(\frac\partial {\partial t}\right) =  \theta_3|_{\sqrt{t^2+z^2}(1,0)}\left(\cos\phi\frac\partial{\partial t} 
+ \sin\phi \frac \partial{\partial z}\right)= \sin\phi\frac 1{m\sqrt{t^2+z^2}} = -\frac{z}{m(t^2+z^2)}.
$$
Similarly, 
$$
\theta_3|_{(t,z)}\left(\frac\partial {\partial z}\right) = \frac t{m(t^2+z^2)}
$$
and therefore the forms
$$
p\theta_3 = \frac1m P(t^2+z^2)(tdz-zdt),\quad q\theta_3 = \frac1m Q(t^2+z^2)(tdz-zdt)
$$ 
extend both smoothly. Moreover the form $dt\wedge \theta_3$ along the normal geodesic is then expressed at a point $(t,z)$ as 
(beware: here $dt$ along the geodesic extends $K$-equivariantly as $dr = \frac 1r(tdt+zdz)$ with $r^2 = t^2+z^2$)
$$\frac 1{\sqrt{t^2+z^2}}(tdt+zdz)\wedge \left(\frac 1{m(t^2+z^2)} (tdz-zdt)\right) = \frac 1{m\sqrt{t^2+z^2}} dt\wedge dz$$
and therefore the form $\ell dt\wedge \theta_3$ extends smoothly, as $\ell(r)= rL(r^2)$.

Hence $\omega$ extends smoothly as well.
The endomorphism $J$ determined by $\omega(\cdot,\cdot)=g(J\cdot,\cdot)$ therefore extends smoothly and the identities 
$J^2=-\mathrm{Id}$ and $N_J=0$ hold everywhere by continuity.

\subsection{The equations}
We now summarize the equations we obtained from the pluriclosed condition \eqref{pluriclosed} and the Bismut Ricci flatness condition (Lemma \ref{lem:BRF}), 
together with the extendability condition over a singular orbit of Type II (Section \ref{sub:extend}):
\begin{equation}\label{system0}
\left\{
\begin{aligned}
p'&=-\frac{2\ell(a+cp)}{\phi},\\
q'&=-2\ell\left(1+\frac{b+cq}{\phi}\right),\\
\phi'&=\frac{2c(\ell^2+p^2+q^2)+4ap+4bq}{\ell},\\
\ell'&=m+2q+
\frac{\ell\phi'}{2\phi}-2\frac{c\ell^2}{\phi},
\end{aligned}\right.
\end{equation}
together with $mc = 2(b+1)$.

\smallskip

If we now plug the third equation into the fourth one, we get the final system 
\begin{equation}\label{system}\left\{
\begin{aligned}
p'&=-\frac{2\ell(a+cp)}{\phi},\\
q'&=-2\ell\left(1+\frac{b+cq}{\phi}\right),\\
\phi'&=\frac{2c(\ell^2+p^2+q^2)+4ap+4bq}{\ell},\\
\ell'&= m+2q+\frac{c(p^2+q^2-\ell^2)+2ap+2bq}{\phi},
\end{aligned}\right.
\end{equation}
together with $mc = 2(b+1)$.

\smallskip

The Taylor expansions at $t=0$ are given by 
\beq\label{tayl1}
p(t)=-\frac{am}{\phi_o}t^2+O(t^4),\qquad
q(t)=-m\left(1+\frac b{\phi_o}\right)t^2+O(t^4),
\eeq
and
\beq\label{tayl2}
\ell(t) = mt + O(t^3),\qquad \phi(t) = \phi_o+\left(2-\frac{2(a^2+b^2)}{\phi_o}\right)t^2+O(t^4).
\eeq
In particular, the value $\phi_o=a^2+b^2$ is distinguished.

\begin{remark} In our system \eqref{system}, the free parameters are 
$$
\phi_o >0,\quad m\in \mathbb N,\, m\geq 1,\quad a,b\in \mathbb R.
$$
However, given $a,b\in\R$ and $c = 2\frac{b+1}{m}$, the vector $Z=a Z_1+bZ_2+cZ_3$ has to satisfy an integrality condition, 
namely $\exp(\mathbb R(A+Z))$ has to be compact. 
\end{remark}

\subsection{Analysis of the equations}\label{sec:phiconst} 
We set 
$$ 
S\coloneqq 2c(\ell^2+p^2+q^2)+4ap+4bq,
$$
so that the third equation in \eqref{system} becomes 
\beq\label{S}\ell\phi' = S.\eeq
Using the fourth equation, we obtain 
$$
\ell' = m + 2q + \frac 1{2\phi}S - \frac{2c}\phi \ell^2.
$$
We now write down a differential equation for $S$, namely we compute 
\[
\begin{split}
S'  &=4c\ell\ell'+4(a+cp)p'+4(b+cq)q'\\
    &=4c\ell\left(m + 2q + \frac 1{2\phi}S - \frac{2c}\phi \ell^2\right)-\frac{8\ell}{\phi}(a+cp)^2-\frac{8\ell}{\phi}(b+cq)(\phi+b+cq)\\
    &=4\ell(cm-2b) + \frac{2c\ell}{\phi}S-\frac{8\ell}{\phi}\left[c^2\ell^2+(a+cp)^2+(b+cq)^2\right]\\
    &= 4\ell(cm-2b) + \frac{2c\ell}{\phi}S-\frac{8\ell}{\phi}\left[a^2+b^2+\frac c2 S\right]\\
    &= 8\ell - \frac{2c\ell}{\phi}S -\frac{8\ell}{\phi}(a^2+b^2) \\
    &=- \frac{2c\ell}{\phi}S + 8\ell \left(1-\frac{a^2+b^2}{\phi}\right),
\end{split}
\]
where we have used the relation $cm-2b=2$.

We now introduce the integrating factor 
$$
I(t)\coloneqq \exp\left(\int_0^t\frac{2c\ell(s)}{\phi(s)} ds\right),
$$
so that, using $S(0)=0$, we obtain
$$
I S = 8 \int_0^t I(s)\ell(s)\left(1-\frac{a^2+b^2}{\phi(s)}\right) ds,
$$
hence 
\beq\label{phi'}
 I\ell \phi' = 8 \int_0^t I(s)\ell(s)\left(1-\frac{a^2+b^2}{\phi(s)}\right) ds.
\eeq
We let 
$$
D\coloneqq a^2+b^2.
$$
\begin{lemma}\label{monotone}
Let $p,q,\ell,\phi$ be solutions to the system \eqref{system} on an interval $[0,T)$, where $\phi>0$ and $\ell>0$ for $t\neq 0$, 
and with the initial conditions 
$$
p(0)=q(0)=\ell(0)=0,\quad \phi(0)=\phi_o > 0.
$$
If $\phi_o>D$ (respectively $\phi_o<D$), then $\phi>D$ and $\phi'>0$ (respectively $\phi<D$ and $\phi'<0$) on $(0,T)$. 
\end{lemma}
\begin{proof} 
Suppose $\phi_o>D$. Then, by continuity, the positivity of  $\ell,\phi$ and $I$ on $(0,T)$ and by \eqref{phi'} we have $\phi>D$ and $\phi'>0$ on a suitable interval $(0,\epsilon)\subset (0,T)$. 
If there exists a first point $\bar t\in (0,T)$ where $\phi(\bar t)=D$, then $\phi'(\bar t)\leq 0$, while $\phi'>0$ in $(0,\bar t)$ by \eqref{phi'}, a contradiction.  
The same argument applies when $\phi_o<D$.
\end{proof}

The next result shows that in the compact case only Type II singular orbits are allowed.
\begin{proposition}\label{prop:noTypeI} 
Let $(M,g,J)$ be a compact non-K\"ahler BHE threefold which is not Bismut-flat. Assume that the group $T^3\times \SU(2)$ acts almost effectively by holomorphic isometries on $M$ with cohomogeneity one. 
Then, $M$ has precisely two complex singular orbits of Type II.
\end{proposition}
\begin{proof} 
We know from Proposition \ref{prop:singiso} that at least one singular orbit is of Type II. 
We choose the normal geodesic $\gamma$ so that this orbit corresponds to $\gamma(0)$, and suppose that the other singular orbit is of Type I. 
We keep the same notation as in Lemma \ref{monotone} with $\g(T)$ singular.
Then, the singular isotropy subalgebra at $\g(T)$ contains the full $\su(2)$ factor. 
In particular, the Killing field $\hat A$ vanishes at $\g(T)$. Now, we recall that the principal isotropy subalgebra is generated by $Z+A$, 
hence $\|\hat A_{\g(t)}\|^2 = \|\hat Z_{\gamma(t)}\|^2$ and (see \eqref{metric})
$$
\|\hat Z_{\gamma(t)}\|^2 =(a+cp)^2+(b+cq)^2+c^2\ell^2,
$$
and by \eqref{S} we get 
$$
\|\hat Z_{\gamma(t)}\|^2 = D + \frac c2 \ell\phi'.
$$
We now distinguish two cases. 
\begin{enumerate}[$\bullet$]
\item  When $c=0$, then $\hat Z_{\g(T)}=0$ implies $D=0$, hence $a=b=c=0$, which is impossible, as $Z\neq 0$.
\item  When $c\neq 0$, then the first expression of $\|\hat Z_{\gamma(t)}\|^2$ above shows that $\ell(T) =0$. 
On the other hand, the function $\phi$ is well defined and smooth in $t=T$, as it is the squared norm of the Killing field $\hat X$. 
This implies that $\phi'$ is bounded and $\hat Z_{\g(T)} = 0$ forces $D=0$. 
By \eqref{phi'}, we then see that $\phi'$ is always positive in $(0,T)$ and $\phi(T)>0$, which contradicts the fact that $\hat X_{\g(T)}=0$. 
\end{enumerate}
\end{proof}

We now suppose that $\g(T)$ is a singular point with singular isotropy subalgebra $\frak k'$ generated by $\frak h$ and another non-zero 
central element $U$. 
Note that $Z_3$ does not necessarily belong to $\frak k'$. The element $U$ can be written as $U = \sum_{i=1}^3 u_iZ_i$, and we have
$$
\|\hat U\|^2 = (u_1+pu_3)^2+(u_2+qu_3)^2+\ell^2u_3^2.
$$
This implies that $\ell(T)=0$. By \eqref{phi'}, we then have 
$$
0 = \int_0^T I(s)\ell(s)\left(1-\frac{a^2+b^2}{\phi(s)}\right) ds  = \int_0^T I(s)\ell(s)\left(\frac{\phi(s)-D}{\phi(s)}\right) ds.
$$
On the other hand, we have $\ell>0, I>0$ in $(0,T)$, and by Lemma \ref{monotone} we know that the function $\phi-D$ has 
a constant sign on $(0,T)$ if $\phi_o \neq D$. Therefore, $\phi_o \neq  D$ gives a contradiction and we have proved that 
$$
\phi_o = D>0.
$$ 
We put $u\coloneqq \phi - D$ and we rewrite \eqref{phi'} as 
\beq\label{eq:u'}
I\ell u' = 8\int_0^t I(s)\ell(s)\frac{u(s)}{\phi(s)} ds,\qquad u(0)=0.
\eeq
Now we recall that $I(0)=1$, $\ell'(0) = m >0$,
so that there exist positive constants $k_1,k_2$ with 
$$
k_1t\leq I\ell,\qquad \frac{I\ell}{\phi}\leq k_2 t,\qquad t\in [0,t_1],
$$
for some positive $0<t_1\leq T$. This implies that for $t\in [0,t_1]$
$$
|u'(t)| \leq \frac{8k_2}{k_1t}\int_0^t s\, |u(s)|\ ds \leq C \int_0^t|u(s)|\, ds,
$$
for some positive constant $C$. We now put 
$$
F(t)\coloneqq \int_0^t|u(s)|\ ds,
$$
so that for $t\in [0,t_1]$ we have 
\[
F'(t)=|u(t)| \leq \int_0^t|u'(s)|\, ds \leq C \int_0^t F(s)\, ds \leq Ct\, F(t)\leq CT\, F(t).
\]
By Gr\"onwall Lemma and $F(0)=0$, we have $F(t)=0$ and $u(t)=0$ for every $t\in [0,t_1]$. 
Let $\bar t>0$ be the maximum of the zeroes of $u$ in $[0,T]$ and suppose $\bar t< T$. Then $u(\bar t) = u'(\bar t)=0$. Differentiating \eqref{eq:u'} gives the following second-order ODE on $(0,T)$ 
\[
u'' + \frac{(I\ell)'}{I\ell} u' -\frac{8}{\phi} u = 0,
\] 
and by uniqueness of the solutions with vanishing initial data at $\bar t$, we obtain $u=0$ everywhere. Hence, $\bar t=T$.    
Therefore, we have proved the following.  
\begin{proposition}\label{D}
If $p,q,\ell,\phi$ is a solution of the system \eqref{system} with the initial conditions $p(0)=q(0)=\ell(0)=0$, $\phi(0)=\phi_o > 0$,
and the corresponding metric extends smoothly over the two singular orbits of Type II, then 
$$
\phi \equiv a^2+b^2 = D >0.
$$
\end{proposition}

\subsection{Explicit solutions to the equations.}\label{sec:Ansatz} 
In this section, we explicitly solve the system of equations \eqref{system} using the fact that $\phi\equiv a^2+b^2.$
By Proposition \ref{D}, this always holds if the BHE theefold $(M,g,J)$ is compact and not Bismut-flat. The system is the following
\begin{equation}\label{system2}\left\{
\begin{aligned}
p'&=-\frac 2D \ell(a+cp),\\
q'&=-\frac 2D\ell\left(D+b+cq\right),\\
\ell'&= m+2q+\frac{c(p^2+q^2-\ell^2)+2ap+2bq}D,\\
0&=2c(\ell^2+p^2+q^2)+4ap+4bq,
\end{aligned}\right.
\end{equation}
together with $mc = 2(b+1)$, and 
\beq 
\label{initial}p(0)=0,\ q(0)=0,\ \ell(0)=0. 
\eeq

The next result allows us to remove the last equation in \eqref{system2}. 
\begin{lemma}\label{lem:algcond} 
The algebraic condition expressing the pluriclosedness of the metric, namely 
$$
2c(\ell^2+p^2+q^2)+4ap+4bq = 0,
$$
is automatically satisfied whenever the functions $p,q,\ell$ satisfy the differential equations in \eqref{system2} together with the initial conditions
\eqref{initial}.
\end{lemma}
\begin{proof} 
This immediately follows from \eqref{S}.
\end{proof}

It then follows that the system becomes 
\[
\left\{
\begin{aligned}
p'&=-\frac 2D \ell(a+cp),\\
q'&=-\frac 2D\ell\left(D+b+cq\right),\\
\ell'&= m+2q+\frac{c(p^2+q^2-\ell^2)+2ap+2bq}D,
\end{aligned}\right.
\]
with $D=a^2+b^2$, $mc = 2(b+1)$, and the initial conditions \eqref{initial}. 
This Cauchy problem has a unique solution and can be simplified using the algebraic condition of Lemma \ref{lem:algcond},  
namely it is equivalent to the Cauchy problem  
\begin{equation}\label{system4}\left\{
\begin{aligned}
p'&=-\frac 2D \ell(a+cp),\\
q'&=-\frac 2D\ell\left(D+b+cq\right),\\
\ell'&= m+2q-\frac{2c}D\ell^2,
\end{aligned}\right.
\end{equation}
with $p(0)=q(0)=\ell(0)=0$.
In order to solve \eqref{system4}, we set 
$$
\l(t)=\int_0^t\ell(s)\ ds,
$$
and we immediately get 
\[
p(t) = \frac ac\left(e^{-2c\l(t)/D}-1\right)\quad (c\neq 0);\qquad p(t) = -\frac{2a}D\l(t),\ (c=0),
\]
\[
q(t) = \frac {D+b}c\left(e^{-2c\l(t)/D}-1\right)\quad (c\neq 0);\qquad q(t) = -2\frac{D+b}D\l(t),\ (c=0)
\]
and we are left with the second order equation 
\beq\label{l''}
\l'' = m + 2q - \frac{2c}D(\l')^2
\eeq
with $\l(0)=\l'(0)=0$. 
We can then put 
$$\mu(t) = e^{\frac{2c}{D}\l(t)},$$
so that for $c\neq 0$ the equation \eqref{l''} becomes 
\[
\begin{split}
\mu''   &= \frac{2c}D\mu(m+2q) 
            = \frac{2c}D\mu\left(m+2 \frac {D+b}c(\mu^{-1}-1)\right)\\
        &= 4\frac {1-D}D\mu + 4 \frac{D+b}D,
\end{split}
\]
where we have used $mc=2(b+1)$. The initial conditions are 
$$
\mu(0)=1,\quad \mu'(0)=0.
$$
We then have three cases: 
\begin{enumerate}[$\bullet$]
\item {\bf Case {\boldmath${D=1}$}.} We have $\mu''= 4(1+b) = 2mc$ and therefore 
$$\mu(t) = 1 + mct^2,$$
hence 
$$\l(t) = \frac1{2c}\log(1+mct^2),\qquad \ell(t) = \frac{mt}{1+mct^2},\qquad (c\neq 0)$$
and $\ell$ is defined for all $t\geq 0$ as $a^2+b^2=1$ and $c\neq0$ imply $b> -1$ and $mc>0$. 
If $c=0$, we have $b=-1,a=0$ and the solution of \eqref{l''} is $\lambda(t) = \frac{m}{2}t^2$, hence
$$
\ell(t)  =  mt,\qquad (c=0).
$$

\item {\bf Case {\boldmath${0<D<1}$}.} Here we remark that $|b|<1$ and therefore $c>0$. Then we obtain 
$$
\mu(t) = 1+\frac{mc}{2(1-D)}
\left(\cosh\left(2\sqrt{\frac{1-D}D} t\right)-1\right),
$$
hence 
$$
\ell(t) = \frac{m\sinh\left(2\sqrt{\frac{1-D}D}t\right)}{2\sqrt{\frac{1-D}D}\mu(t)}.
$$
We note that $\mu(t)\geq 1$ for all $t\geq 0$ and therefore the solution $\ell$ is defined for all $t\geq 0$ and 
$$\ell(t) > 0, \quad t>0.$$

\item {\bf Case {\boldmath${D>1}$}.} We have 
$$
\mu(t) = 1 + \frac{mc}{2(D-1)}\left(1 - \cos\left(2\sqrt{\frac{D-1}D}t\right)\right),
$$
so that if $c\neq 0$
$$
\ell(t) = \frac D{2c}\frac{\mu'}\mu = \frac{m\sin\left(2\sqrt{\frac{D-1}D} t\right)}{2\sqrt{\frac{D-1}D}\mu(t)},
$$
while for $c=0$ we have $b=-1$ and we solve the equation \eqref{l''} 
$$
\l(t) =\frac{mD}{4(D-1)} \left(1-\cos\left(2t \sqrt{\frac{D-1}{D}}\right)\right), 
$$
hence 
$$\ell(t) = \frac m2\sqrt{\frac D{D-1}}\sin\left(2t\sqrt{\frac{D-1}D}\right).$$

\end{enumerate}

\subsection{Curvature properties of the solutions with constant {\boldmath$\phi$}}\label{sec:scal}
We now discuss some useful curvature properties of the BHE metric $g$ under the Ansatz $\phi\equiv D = a^2+b^2$. 
These will be useful in the subsequent analysis. 

\begin{lemma}\label{flat} 
Let $g$ be a BHE metric as in \eqref{metric} with the normalization of Lemma \ref{lem:BRF}, the smooth extension at a Type II orbit and 
$\ell'(0)=m$.  
Suppose $\phi \equiv D>0$ with $D=a^2+b^2$. 
At the singular point $x_o=\gamma(0)$, the Ricci tensor $\Ric^B$ of the Bismut connection satisfies 
$$
\Ric^B(\g'(0),\g'(0))|_{x_o} = \frac{2mc}D.
$$
\end{lemma}

\begin{proof} 
A standard formula (see e.g. \cite{DW}) shows that the Ricci tensor of the metric $g$ along the normal geodesic $\g$ with $\g'(t)=\xi$ is given by 
$$
\Ric^g(\xi,\xi) = -\frac{\ell''}{\ell}- \frac{(p')^2+(q')^2}{2\ell^2}. 
$$
Moreover, we know that (see, e.g., \cite{GFS})
$$
\Ric^B(\xi,\xi) = \Ric^g(\xi,\xi) -\frac 14\cH^2(\xi,\xi).
$$
By \eqref{H} we have that 
$$\imath_\xi \cH = \theta_3\wedge(p'\eta_1+q'\eta_2)$$
and we recall that $\ell\theta_3,\eta_1,\eta_2$ are orthonormal. This implies that 
$$
\frac 14\cH^2(\xi,\xi) = \frac 12 |\imath_{\xi}\cH|^2 = 
\frac 12 \frac{(p')^2+(q')^2}{\ell^2},
$$
and therefore 
$$
\Ric^B(\xi,\xi) = -\frac{\ell''}{\ell}- \frac{(p')^2+(q')^2}{\ell^2}.
$$
Our claim now follows from the expression of $\Ric^B(\xi,\xi)$ and the expansions of the functions $\ell,p,q$ at $t=0$, 
which can be easily recovered from the system \eqref{system4}
$$
p(t) = -\frac{ma}Dt^2 + O(t^4),\quad q(t)= -m\frac{D+b}Dt^2 + O(t^4),\quad \ell(t)= mt - \left(\frac{2m(D-1)}{3D} + \frac{m^2c}D\right)t^3 + O(t^5).
$$
\end{proof}

We now compute the scalar curvature of the $G$-invariant metric $g$ given in \eqref{metric} on the regular part $M_{\mathrm{reg}} \cong (0,T)\times G/H$. 
Recall that $g|_{M_{\mathrm{reg}}}= dt^2+g_t$, where for every $t\in(0,T)$
\[
g_t = \eta_1^2 + \eta_2^2 + \ell^2\,\theta_3^2 + \phi (x^2+y^2)
\]
is a $G$-invariant metric on the principal orbit $P_t\coloneqq G\cdot \gamma(t) \cong G/H$.  

Consider the unit normal vector field $\xi$ and let $L_t = \nabla^{LC}\xi$ denote the shape operator of $P_t$. 
Then, from \cite{DW,EW}, we know that
\[
L_t = \frac12 g_t^{-1}g'_t,
\]
and that the scalar curvature of $g$ along the normal geodesic $\gamma(t)$ is given by
\beq\label{eq:scal}
\mathrm{Scal}_g = \mathrm{Scal}_{g_t} - 2\,\mathrm{tr}(L_t') - \mathrm{tr}(L_t^2) - (\mathrm{tr}(L_t))^2. 
\eeq
 
By \cite[(7.39)]{Bes}, the scalar curvature $\mathrm{Scal}_{g_t}$ of the principal orbit $P_t$ has the following expression with respect to a 
$g_t$-orthonormal basis
$\{E_1,\ldots,E_5\}$ of $\mathfrak{m} \coloneqq \mathfrak{p}_o\oplus\mathfrak{p}_1$:
\[
\mathrm{Scal}_{g_t} = 	-\frac14 \sum_{1\leq i,j\leq 5}\left\|[E_i,E_j]_{\mathfrak{m}}\right\|^2_{g_t} 
					-\frac12 \sum_{i=1}^5 B(E_i,E_i), 
\]
where $B$ is the Cartan-Killing form of $\mathfrak{g}= \R^3\oplus\mathfrak{su}(2)$. 
If we choose the orthonormal basis $\{E_1,\ldots,E_5\} = \left\{Z_1,Z_2,\tfrac{1}{\ell}W,\tfrac{1}{\sqrt{\phi}}X,\tfrac{1}{\sqrt{\phi}}Y\right\}$, 
we have
\[
\sum_{1\leq i,j\leq 5}\left\|[E_i,E_j]_{\mathfrak{m}}\right\|^2_{g_t}  = 2\left\|[E_4,E_5]_{\mathfrak{m}}\right\|^2_{g_t} = \frac{8}{\phi^2}\left\|Z\right\|^2_{g_t}  
                                                                        = \frac{8}{\phi^2}\left((a+pc)^2+(b+qc)^2+c^2\ell^2\right),
\]
since the projection of $A$ onto $\mathfrak{m}$ is $-Z = -\left((a+pc)Z_1+(b+qc)Z_2+cW\right)$, and 
\[
\sum_{i=1}^5 B(E_i,E_i) = \frac{1}{\phi}\left(B(X,X)+B(Y,Y)\right) = \frac{1}{\phi}\left(\mathrm{tr}(\mathrm{ad}_X\circ\mathrm{ad}_X) + \mathrm{tr}(\mathrm{ad}_Y\circ\mathrm{ad}_Y)\right) = -\frac{16}{\phi},
\]
so that
\[
\begin{split}
\mathrm{Scal}_{g_t} &= -\frac{2}{\phi^2}\left((a+pc)^2+(b+qc)^2+c^2\ell^2\right) + \frac{8}{\phi}\\
                    		&= -\frac{2}{\phi^2}\left(a^2+b^2+2apc+2bqc+c^2(\ell^2+p^2+q^2)\right) + \frac{8}{\phi}.
\end{split}
\]
Using the pluriclosed condition $\ell\phi' = 2c(\ell^2+p^2+q^2)+4ap+4bq$, this expression can be rewritten as follows
\beq\label{eq:scalt}
\mathrm{Scal}_{g_t} = \frac{8}{\phi} -2\,\frac{a^2+b^2}{\phi^2} -c\,\frac{\ell\phi'}{\phi^2}.
\eeq

\medskip

We now compute the remaining contributions in the RHS of \eqref{eq:scal}. We first observe that
\[
\det(g_t) = (\ell\phi)^2.
\]
Moreover, we have
\[
g'_t =  2p'\eta_1\theta_3 +2q'\eta_2\theta_3+ 2\ell\ell'\,\theta_3^2 + \phi' (x^2+y^2), 
\]
so that the matrix representing $L_t$ with respect to the considered $g_t$-orthonormal basis is
\[
L_t =   \begin{pmatrix}
            0   &   0   &   p'/2\ell    &   0   &   0   \\
            0   &   0   &   q'/2\ell    &   0   &   0   \\
            p'/2\ell   &   q'/2\ell    &   \ell'/\ell    &   0   &   0\\
            0   &   0   &   0    &   \phi'/2\phi   &   0   \\
            0   &   0   &   0    &   0   &   \phi'/2\phi   
        \end{pmatrix}.
\]
We then have:
\begin{itemize}
\item 
$ \displaystyle
\mathrm{tr}(L_t) 	= \frac{1}{2} \mathrm{tr}(g_t^{-1}g_t') = \frac{1}{2} \frac{(\det(g_t))'}{\det(g_t)} 
                    = \frac12 \frac{2\ell\phi(\ell'\phi+\ell\phi')}{\ell^2\phi^2} = \frac{\ell'}{\ell} + \frac{\phi'}{\phi}
$;
\vspace{0.2cm}

\item 
$\displaystyle
-2\mathrm{tr}(L_t') = -2(\mathrm{tr}(L_t))' = -2\left(\frac{\ell'}{\ell} + \frac{\phi'}{\phi}\right)'
$;
\vspace{0.2cm}

\item
$\displaystyle
\tr(L_t^2) 	= \left(\frac{\ell'}{\ell}\right)^2 + \frac12 \left(\frac{\phi'}{\phi}\right)^2 + \frac12 \left(\frac{p'}{\ell}\right)^2 + \frac12 \left(\frac{q'}{\ell}\right)^2
$.
\end{itemize}

\smallskip

Plugging these expressions together with \eqref{eq:scalt} into \eqref{eq:scal}, 
we obtain the most general expression of the scalar curvature along the normal geodesic $\gamma$ in our setting. 
When $\phi$ is constant, we have the following. 
\begin{proposition}\label{prop:scla0}
Let $(M,g,J)$ be a $T^3\times\SU(2)$-invariant BHE threefold of cohomogeneity one, with the metric given by \eqref{metric} and the normalization 
of Lemma \ref{lem:BRF}. 
Suppose that $g$ extends smoothly over a Type II singular orbit and that $\phi$ is constant. 
Let $\gamma$ be a unit-speed normal geodesic with $\gamma(0)$ on that singular orbit. 
Then, $\phi=a^2+b^2$ and along a regular portion of $\gamma$ 
\beq\label{scal}
\mathrm{Scal}_g = \frac{6}{a^2+b^2} -2\frac{\ell''}{\ell} 
                  - \frac12 \left(\frac{p'}{\ell}\right)^2 - \frac12 \left(\frac{q'}{\ell}\right)^2.\eeq
In particular, 
\[
\mathrm{Scal}_g(\gamma(0)) = \lim_{t\to0^+} \mathrm{Scal}_g(\gamma(t)) = 6 + 20\,\frac{b+1}{a^2+b^2}.
\]
\end{proposition}

\begin{proof}
    The expression of the scalar curvature easily follows from the computations above, the condition $\phi'=0$ and the identity 
    $\left(\frac{\ell'}{\ell} \right)'+\left(\frac{\ell'}{\ell} \right)^2 = \frac{\ell''}{\ell}$. 
    
    As for the limit, recall the expansions of the functions $p$, $q$, $\ell$ at $t=0$
    \[
    p(t) = -\frac{ma}Dt^2 + O(t^4),\quad 
    q(t)= -m\frac{D+b}Dt^2 + O(t^4),\quad
    \ell(t)= mt - \left(\frac{2m(D-1)}{3D} + \frac{m^2c}D\right)t^3 + O(t^5).
    \]
    Using these, we obtain
    \[
    \lim_{t\to0^+} \mathrm{Scal}_g  = \frac{6}{D} 
                                        +8\frac{D-1}{D} + 12\frac{mc}D
                                        -2\frac{a^2}{D^2}
                                        -2 \frac{(D+b)^2}{D^2}
                                    = \frac{6D^2+4D(3mc-b-1)}{D^2}
                                    = 6 + 20\,\frac{b+1}{D}, 
    \]
    since $3mc = 6(b+1)$. 
\end{proof}

\subsection{Discussion of the explicit solutions: the non-compact case {\boldmath$D\leq 1$}.}\label{sec:completex}
We now look at each explicit solution we have singled out in the previous sections and study their geometric realizations when $D\leq 1$. 
Note that all the explicit solutions exist on the whole real line, so that the metrics which they define are complete. 

We distinguish three different subcases. \par\medskip

\begin{enumerate}[$\bullet$]
\item  {\bf Case {\boldmath$D=1$, $c=0$}.} In this case we have $b=-1,a=0$ and the solutions are 
$$
p= q \equiv 0,\qquad \ell = mt,
$$
and the metric and the fundamental form are 
\[
\begin{split}
g   &= dt^2 + m^2t^2\theta_3^2 + \theta_1^2 +\theta_2^2 + x^2 + y^2,\\
\o  &= \theta_1\wedge \theta_2 + mt dt\wedge \theta_3 + \Omega,
\end{split}
\]
with 
$$d\theta_1=d\theta_3 = 0,\ d\theta_2 = -2\Omega.$$
If we take local coordinates $z,\psi$ with $\theta_1 = dz$, $m\theta_3 = d\psi$ then 
$$
dt^2 + t^2d\psi^2,\quad tdt\wedge d\psi
$$
correspond to a flat factor $\R^2$, while the other factor is 
$$
dz^2 + \theta_2^2+ x^2+y^2,\qquad dz\wedge\theta_2 + \Omega
$$
and corresponds locally to the product $\mathbb R\times S^3$, where $S^3$ is the standard sphere of sectional curvature $1$, with the standard Hopf structure. 
Thus, this case is Bismut-flat. \par\medskip

\item {\bf Case {\boldmath$D=1,\ c\neq 0$}.} Note that in this case $c>0$, as $|b|\leq 1$. The solutions are 
\beq\label{D=1,c not 0}
p(t) = -\frac{amt^2}{1+mct^2},\qquad q(t)=-\frac{m^2ct^2}{2(1+mct^2)},\qquad  \ell(t)= \frac{mt}{1+mct^2}.
\eeq
Note that these solutions are defined for all $t$. As $D=a^2+b^2=1$, we have the constraint 
$$
0< mc \leq 4, \qquad m\in\mathbb N,\ m\geq 1.
$$
For each choice of parameters $(m,c)$ as above, we have $b=\frac{mc}2-1$, $a=\pm\sqrt{1-b^2}$ and the corresponding solutions as in \eqref{D=1,c not 0}. 

By Lemma \ref{flat}, we know that at the singular point $x_o=\gamma(0)$
\[
\Ric^B(\g'(0),\g'(0))|_{x_o} = \frac{2mc}D.
\]
Therefore, these metrics are not Bismut-flat.
\end{enumerate}

\begin{enumerate}[$\bullet$]
\item {\bf Case {\boldmath$0<D<1$}.} We recall that in this case $c>0$. Again we obtain complete, non-Bismut-flat, BHE metrics, given by the solutions 
\beq\label{D<1}
\begin{split}
p(t) &= -am\frac{\cosh\left(2\sqrt{\frac{1-D}D}t\right)-1}{2(1-D) + mc\left(\cosh\left(2\sqrt{\frac{1-D}D}t\right)-1\right)},\\
q(t) &= -m(D+b)\frac{\cosh\left(2\sqrt{\frac{1-D}D}t\right)-1}{2(1-D) + mc\left(\cosh\left(2\sqrt{\frac{1-D}D}t\right)-1\right)}.\\
\ell &= m\sqrt{D(1-D)}\frac{\sinh\left(2\sqrt{\frac{1-D}D}t\right)}{2(1-D) + mc\left(\cosh\left(2\sqrt{\frac{1-D}D}t\right)-1\right)}.
\end{split}
\eeq
\end{enumerate}

We can summarize the results we have obtained so far by stating the following theorem. 
\begin{theorem}\label{thm:noncptex}
Let $G=T^3\times \SU(2)$. Fix a basis $\{Z_1,Z_2,Z_3\}$ of the center of $\frak g$ with $Z_3$ primitive integral and a nonzero element $A$ in $\su(2)$ 
from a basis satisfying \eqref{su2}. Let 
$$
m\in \mathbb N_{>0},\qquad 0<D = a^2+b^2\leq 1,\qquad c= 2\frac{b+1}m,\qquad b>-1.$$
Suppose that there is a closed subgroup $K\subset G$ whose Lie algebra is 
$$\frak k = \operatorname{Span}(A+aZ_1+bZ_2+cZ_3,Z_3)$$
and suppose $K$ admits a representation $\nu_m:K\to O(2)$ so that 
$$\nu_m(\exp(sZ_3)) = e^{ism},\qquad 
\ker d\nu_m= \mathbb R\cdot (A+aZ_1+bZ_2+cZ_3).$$
Then the manifold $M\coloneqq  G\times_K \mathbb C$ admits a $G$-invariant complex structure $J_{m,a,b}$ and a Hermitian, 
complete metric $g_{m,a,b}$ which is BHE and is not Bismut-flat.
\end{theorem}

\begin{remark}\label{rem:singorb}
The singular orbit $G/K$ is always finitely covered by $S^3\times S^1$, but the complex structures are not necessarily fixed. 
Indeed, let $\mathrm{pr}:K\to Z(G)$ be the projection of $K$ on the center and let $\tilde K\coloneqq \mathrm{pr}(K^o)$. 
Then $\tilde K\cong T^2$ and we can find a closed circle $C\subset Z(G)$ with $\frak z = {\tilde{\frak k}}\oplus \frak c$. 
Now the subgroup $Q\coloneqq  \SU(2)\times C\cong S^3\times S^1$ acts transitively on $G/K$ as $Q\cap K$ is finite and $\dim Q=\dim G/K$.
Therefore, the homogeneous space $G/K$ is finitely covered by $S^3\times S^1$. 
Since $H^2(S^3\times S^1,\mathbb{Z}) = \{0\}$, the manifold $M$, which is a rank $2$ vector bundle over $G/K$, is finitely covered by $S^3\times S^1 \times \R^2$.
\end{remark}

\begin{remark} 
We note that the assumptions on $K$ and its slice representation can be realized for every choice of $m,a,b$ in the theorem. 
Indeed, put
$$
U\coloneqq aZ_1+bZ_2+2\frac{b+1}mZ_3,\quad C\coloneqq -bZ_1+aZ_2,
$$
and note that $\{U,C,Z_3\}$ is a basis of the center $\frak z$ as $a^2+b^2>0$. Then we can consider the center $Z(G)$ as the torus 
$$T^3 = \frak z/2\pi \operatorname{Span}_{\mathbb Z}(U,C,Z_3).$$
Then the group 
$$K \coloneqq \{\exp(s_1(A+U))\exp(s_2Z_3)|\ s_1,s_2\in [0,2\pi]\}$$
is compact and admits the representation $\nu_m$ given by 
$$\nu_m(\exp(s_1(A+U))\exp(s_2Z_3)) = e^{ims_2}.$$
In this case, we have $G/K\cong S^3\times S^1$ and $G\times_K\mathbb R^2 \cong S^3\times S^1\times \mathbb R^2$.
\end{remark}

We can now show that the metrics constructed in Theorem \ref{thm:noncptex} are not locally homogeneous. 

\begin{proposition}
Let $(M= G\times_K\C,g_{m,a,b}, J_{m,a,b})$ be a complete BHE threefold from Theorem \ref{thm:noncptex}, where $G=T^3\times\SU(2)$. Then, 
\[
\lim_{t\to+\infty}\mathrm{Scal}_g(\gamma(t)) = \frac{6}{a^2+b^2}. 
\]
\end{proposition}
\begin{proof} 
Using the expressions of $p(t)$, $q(t)$, $\ell(t)$ in \eqref{D=1,c not 0} when $D=1$ and in \eqref{D<1} when $0<D<1$, we obtain that in both cases
\[
\lim_{t\to+\infty}\frac{\ell''}{\ell} = \lim_{t\to +\infty} \frac{p'}\ell = \lim_{t\to +\infty} \frac{q'}\ell = 0,
\]
hence by \eqref{scal} we have
\[
\lim_{t\to +\infty} \mathrm{Scal}_g  = \frac{6}{D} = \frac{6}{a^2+b^2}.
\]
\end{proof}

Comparing this result with the expression of the scalar curvature at the singular point $\gamma(0)$ (Proposition \ref{prop:scla0}), 
we immediately obtain the following. 

\begin{corollary}\label{cor:scalcompl}
    Every complete BHE metric $g_{m,a,b}$ constructed in Theorem\ref{thm:noncptex} has non-constant scalar curvature. 
    In particular, it is not locally homogeneous.
\end{corollary}
\begin{proof} We compare the scalar curvature at the singular orbit with its limit when $t\to +\infty$.
We obtain 
$$ \mathrm{Scal}_g(\gamma(0)) - \lim_{t\to +\infty} \mathrm{Scal}_g  = 6 + 20 \frac{b+1}D - \frac 6D =  \frac{6D+20b+14}D.$$
We recall that $0<D\leq 1$ and $c\neq 0$, so that $b>-1$. Then 
$$6D+20b+14 = 6a^2+6(b+1)^2 + 8(b+1) >0.$$
\end{proof}

\subsection{Discussion of the explicit solutions: the compact case {\boldmath$D>1$}.}\label{sec:cptsol}
When $D>1$, we distinguish two cases:
\begin{enumerate}[$\bullet$]
\item When $c=0$, the solutions are 
\[
\begin{split} 
p(t)&= -\frac{am}{2(D-1)}\left(1-\cos\left(2\sqrt{\frac{D-1}D}t\right)\right),\\
q(t)&= -\frac m2\left(1-\cos\left(2\sqrt{\frac{D-1}D}t\right)\right),\\
\ell(t)&= \frac m2 \sqrt{\frac D{D-1}}\sin\left(2\sqrt{\frac{D-1}D}t\right).
\end{split}
\]

\item When $c\neq 0$, we have 
\[
\begin{split}
\mu(t)&= 1+\frac{mc}{2(D-1)}\left(1-\cos\left(2\sqrt{\frac{D-1}D}t\right)\right),\\
p(t)&= \frac ac\left(\frac 1{\mu(t)}-1\right),\\
q(t)&= \frac{D+b}c\left(\frac 1{\mu(t)}-1\right),\\
\ell(t)&= \frac m2 \sqrt{\frac D{D-1}}\, \frac{1}{\mu(t)}\, \sin\left(2\sqrt{\frac{D-1}D}t\right).
\end{split}
\]
\end{enumerate}
Note that the function $\mu(t)$ never vanishes for all $t$: indeed, if $c>0$ then $\mu(t)\geq 1$, while for $c<0$ a direct inspection shows that $\mu(t)=0$ 
means that $\cos\left(2\sqrt{\frac{D-1}D}t\right) = \frac{mc+2D-2}{mc}<-1$, a contradiction.

Consequently, in both cases, the solution is defined on the interval $(0,T)$ where \(\ell(t)>0\), namely we have 
\[
T = \frac \pi 2\sqrt{\frac D{D-1}}, 
\]
and extends smoothly across the singular orbit at $t=0$. 

\medskip

We now rule out the case $c=0$. 
\begin{lemma}\label{c=0} If $D>1$, the $G$-invariant metrics are  Bismut-flat if and only if $c=0$. \end{lemma}
\begin{proof} 
When $c=0$, then $b=-1$ and $D=1+a^2$, with $a\neq0$. Set
$$
\beta_1=\frac{\theta_1+a\theta_2}{\sqrt D},\qquad
\beta_2=\frac{-a\theta_1+\theta_2}{\sqrt D},\qquad
\kappa=\frac{2|a|}{\sqrt D}.
$$
The structure equations are
$$
d\beta_1=0,\qquad
d\beta_2=-2\sqrt D\,\Omega,\qquad 
d\theta_3=0.
$$
Since $q=ap$, we have
$$
\eta_1^2+\eta_2^2=\beta_2^2+\widehat\beta_1^{\,2},\qquad 
\widehat\beta_1 \coloneqq \beta_1+\sqrt D\,p\,\theta_3=\beta_1-\frac{m\sqrt D}{2a}(1-\cos(\kappa t))\theta_3.
$$
Hence
$$
g=g_1+g_2,\qquad 
g_1=\beta_2^2+D(x^2+y^2),\qquad 
g_2=dt^2+\widehat\beta_1^{\,2}+\ell^2\theta_3^2.
$$

The structure equations show that the complementary orthogonal distributions
\[
\mathcal D_1=\ker dt\cap\ker\beta_1\cap\ker\theta_3,
\qquad
\mathcal D_2=\ker\beta_2\cap\ker x\cap\ker y
\]
are integrable. Moreover, $g_1$ is invariant along $\mathcal D_2$, since the variation of $x,y$ in these directions is a rotation,
while $g_2$ is invariant along $\mathcal D_1$, since $d\beta_1=d\theta_3=0$ and its coefficients depend only on $t$.
Thus $g=g_1+g_2$ is locally a Riemannian product. The torsion also splits according to this local product decomposition.
Indeed, since $q'=ap'$, we have
$$
\begin{aligned}
\cH &=dt\wedge\theta_3\wedge(p'\eta_1+q'\eta_2)
  +2(a\eta_1-\eta_2)\wedge\Omega\\
&=\sqrt D\,p'\,dt\wedge\theta_3\wedge\widehat\beta_1
  -2\sqrt D\,\beta_2\wedge\Omega.
\end{aligned}
$$
Using $p'=-2a\ell/D$, we obtain
$$
\cH=-\frac2{\sqrt D}\,\mathrm{vol}_{g_1}
  +\frac{2a}{\sqrt D}\,\mathrm{vol}_{g_2},
$$
where $\mathrm{vol}_{g_1}=\beta_2\wedge D\Omega$ and $\mathrm{vol}_{g_2} =dt\wedge\widehat\beta_1\wedge\ell\theta_3$. 
Thus $\nabla^B=\nabla^g+\frac12\,\cH$ is the product of two metric connections on the three-dimensional factors.
We now fix two tangent vectors $U,W$ in the first factor and consider the curvature endomorphism $R^B(U,W)$, that acts trivially
on the second factor and is skew-adjoint on the first.
Therefore it has real rank at most two. On the other hand, it belongs to $\su(3)$. 
A nonzero skew-Hermitian endomorphism of real rank two has complex rank one and exactly one nonzero eigenvalue, so its complex trace cannot vanish. 
Therefore $R^B(U,W)=0$. The same argument applies to vectors tangent to the second factor, while the mixed curvature vanishes because 
the connection is a product. It follows that $R^B=0$.

When $c\neq 0$, the metric is not Bismut-flat by Lemma \ref{flat}, as the expansions for $p,q,\ell$ hold also in this case. 
\end{proof}
We now continue our discussion recalling that $\ell^2=g(W, W)$ along the geodesic, i.e., the central element
$$
Z' = Z_3-p(T)Z_1-q(T)Z_2
$$
belongs to the isotropy subalgebra $\frak k'$ at the point $\gamma(T)$. Using the expressions of $p,q$, we obtain 
\[
Z'= Z_3 + \frac{am}{D+2b+1}Z_1 + \frac{m(D+b)}{D+2b+1}Z_2
\]
so that the isotropy subalgebra $\frak k'$ is given by 
$$\frak k' = \frak h \oplus \mathbb R Z'.$$

In order to construct a smooth manifold of cohomogeneity one with prescribed group diagram, we need to require the element $Z'$ to generate a closed subgroup. 
Our result is the following.
\begin{theorem}\label{Tcompact}
Let $G=T^3\times \mathrm{SU}(2)$. 
Fix a basis $\{Z_1,Z_2,Z_3\}$ of the center of $\frak g$ with $Z_3$ primitive integral and a nonzero element $A$ in $\su(2)$ 
from a basis satisfying \eqref{su2}.
Let
$$m\in\mathbb N_{>0},\qquad
a,b\in\mathbb R,\qquad
D=a^2+b^2>1,\qquad
c=\frac{2(b+1)}{m}\neq 0.$$
We define the subalgebra
$$
\mathfrak h=\mathbb R(A+aZ_1+bZ_2+cZ_3),
$$
and the central element
$$
Z'=Z_3+\frac{ma}{\Delta}Z_1+\frac{m(D+b)}{\Delta}Z_2, 
\qquad
\Delta=D+2b+1=a^2+(b+1)^2>0.
$$
Suppose there exist closed subgroups
$H, K,K'$, with $H\subseteq K\cap K'$,  satisfying the following conditions:
\begin{enumerate}[$(1)$]
\item
$\operatorname{Lie}(H)=\mathfrak h$, and 
$\mathbb RZ'$ integrates to a closed circle with primitive integral generator $U=\lambda Z'$, where $\lambda>0$;
\item
the singular isotropy algebras are
$$\frak k =\mathfrak h+\mathbb RZ_3,
\qquad
\mathfrak k'=\mathfrak h+\mathbb RU;$$
\item there are orthogonal representations
$\nu:K\to\mathrm{O}(2)$, $\nu':K'\to \mathrm{O}(2)$ which are transitive on the unit circle, with principal isotropy $H$.
The slice weights of $Z_3$ and $U$ are $m$ and $m'\in\mathbb N_{>0}$, respectively, with
$$m'=\lambda\,\frac{m(D-1)}{\Delta}.$$
\item
The resulting cohomogeneity-one $G$-action is almost effective.
\end{enumerate}
Then the compact manifold
$$ 
M=\left(G\times_{K}\mathbb D^2\right)\underset{G/H}{\cup} \left(G\times_{K'}\mathbb D^2\right),
$$
where the boundaries are identified by the identity map of $G/H$, admits a smooth $G$-invariant integrable complex structure $J_{m,a,b}$ and a $G$-invariant Hermitian metric $g_{m,a,b}$ which is BHE and is not Bismut-flat. 
\end{theorem}

\begin{proof}
We verify the extension at $t=T$, the extension at $t=0$
being already established. Put
$$
s\coloneqq T-t,\qquad p_T \coloneqq p(T),\qquad q_T \coloneqq q(T),\qquad \kappa \coloneqq 2\sqrt{\frac{D-1}D},
$$
and recall that 
$$
U=\lambda(Z_3-p_TZ_1-q_TZ_2).
$$
The forms dual to the basis $Z_1,Z_2,U$ are
$$
\tilde\theta_1=\theta_1+p_T\theta_3,\qquad
\tilde\theta_2=\theta_2+q_T\theta_3,\qquad
\tilde\theta_3=\lambda^{-1}\theta_3.
$$
If we now define
$$
\tilde p(s)\coloneqq\lambda\left(p(T-s)-p_T\right),\qquad
\tilde q(s)\coloneqq\lambda\left(q(T-s)-q_T\right),\qquad
\tilde\ell(s)\coloneqq\lambda\ell(T-s),
$$
then
$$
\begin{aligned}
g={}&ds^2+
(\tilde\theta_1+\tilde p\,\tilde\theta_3)^2+
(\tilde\theta_2+\tilde q\,\tilde\theta_3)^2
+\tilde\ell^{\,2}\tilde\theta_3^2
+D(x^2+y^2),\\
\omega={}&
(\tilde\theta_1+\tilde p\,\tilde\theta_3)
\wedge
(\tilde\theta_2+\tilde q\,\tilde\theta_3)
-\tilde\ell\,ds\wedge\tilde\theta_3+D\Omega.
\end{aligned}
$$
Since $\kappa T=\pi$, we have
$$
\cos(\kappa(T-s))=-\cos(\kappa s),\qquad
\sin(\kappa(T-s))=\sin(\kappa s).
$$
The explicit formulas therefore show that 
$$
\tilde p(s)=s^2\tilde P(s^2),\qquad
\tilde q(s)=s^2\tilde Q(s^2),\qquad
\tilde\ell(s)=s\tilde L(s^2),
$$
where $\tilde P,\tilde Q,\tilde L$ are analytic in a neighborhood of $s=0$. 
Moreover,
\[
\tilde L(0)
=-\lambda\ell'(T)
=\lambda\frac{m(D-1)}{\Delta}
=m'.
\]
Thus, with respect to the inward normal parameter $s$ and the central basis $Z_1,Z_2,U$, the coefficients satisfy exactly the smoothness conditions 
established in Section \ref{sub:extend}, with $m$ replaced by $m'$. It follows that $g$ extends smoothly and positively over the singular orbit $G/K'$.

The extension of $\omega$ follows from the same argument.
Indeed, the forms $\tilde\theta_1,\tilde\theta_2,\Omega$ are horizontal for $\mathfrak k'$ and invariant under $K'$, and hence descend to $G/K'$. 
Their pullbacks under the normal bundle projection agree with the corresponding forms on the regular part, providing their smooth extensions.
Here the invariance under the group $K'$ follows from the invariance under $K'^o$ and $H$, since $K'/H$ is connected and therefore $K'=K'^oH$.

Furthermore, the slice computation in Section \ref{sub:extend} applies with $Z_3,\theta_3,m,t$ replaced by $U,\tilde\theta_3,m',s$, respectively.
Since
\[
\tilde p=s^2\tilde P(s^2),\qquad
\tilde q=s^2\tilde Q(s^2),\qquad
\tilde\ell=s\tilde L(s^2),
\]
it shows that
\[
\tilde p\,\tilde\theta_3,\qquad
\tilde q\,\tilde\theta_3,\qquad
\tilde\ell\,ds\wedge\tilde\theta_3
\]
extend smoothly. The minus sign in the last term of
\[
\omega=
(\tilde\theta_1+\tilde p\,\tilde\theta_3)
\wedge
(\tilde\theta_2+\tilde q\,\tilde\theta_3)
-\tilde\ell\,ds\wedge\tilde\theta_3+D\Omega
\]
does not affect this conclusion.
Consequently, $\omega$ and the endomorphism $J$ determined by $\omega(\cdot,\cdot)=g(J\cdot,\cdot)$ extend smoothly.
The equations $J^2=-\mathrm{Id}$, $N_J=0$, $dd^c\omega=0$ and $\rho^B=0$ hold everywhere by continuity.
\end{proof}

We now show that the metrics constructed in Theorem \ref{Tcompact} are not locally homogeneous. 
\begin{proposition}
    Let $(M,g_{m,a,b},J_{m,a,b})$ be a compact BHE threefold from Theorem \ref{Tcompact}. Then, 
    \[
    \lim_{t\to T}\mathrm{Scal}_g(\gamma(t)) =  6 - 20\frac{(b+1)(D-1)}{D(a^2+(b+1)^2)},
    \]
    where $D = a^2+b^2>1$.
\end{proposition}
\begin{proof}
Using the expressions of $\mu(t)$, $p(t)$, $q(t)$, $\ell(t)$ when $D>1$ and $c\neq0$, we see that
\[
\frac{p'}{\ell} = -\frac{2a}{D\mu},\qquad \frac{q'}{\ell} = -2\frac{D+b}{D\mu}.
\]
Moreover,
$
\ell' = m + 2q -\tfrac{2c\ell^2}{D}, 
$
so that 
\[
\frac{\ell''}{\ell} = 2\frac{q'}{\ell} - \frac{4c\ell'}{D}.
\]
Since $\ell(T)=0$ and
\[
\mu(T) =  1 + \frac{mc}{D-1} = \frac{D-1+mc}{D-1} = \frac{a^2 +(b+1)^2}{D-1},
\]
we obtain from \eqref{scal}
\[
\lim_{t\to T^-}\mathrm{Scal}_g(\gamma(t)) = \frac{6}{D} + 8\frac{D+b-mc}{D\mu(T)} -2 \frac{a^2+(D+b)^2}{(D\mu(T))^2} = 6 - 20\frac{(b+1)(D-1)}{D(a^2+(b+1)^2)}.
\]
\end{proof}

\begin{corollary}\label{cor:scalcpt}
    Every compact BHE metric constructed in Theorem\ref{Tcompact} has non-constant scalar curvature. 
    In particular, it is not locally homogeneous.    
\end{corollary}
\begin{proof}
    We compare the scalar curvature at the singular orbits. We obtain 
    $$ 
    \mathrm{Scal}_g(\gamma(0)) - \mathrm{Scal}_g(\gamma(T))  = 20 \frac{b+1}D  + 20\frac{(b+1)(D-1)}{D(a^2+(b+1)^2)} = 20\frac{b+1}{D}\left(1+\frac{D-1}{a^2+(b+1)^2}\right). 
    $$
    Since $D>1$ and $c\neq 0$, so that $b+1\neq0$, this difference is not zero. 
\end{proof}

\medskip

The next example shows that the hypotheses of Theorem \ref{Tcompact} are far from being too restrictive. 
Moreover, it will also allow us to prove a further remarkable result on the possibility of smoothly deforming a Bismut-flat Samelson structure through non-K\"ahler BHE structures that are not Bismut-flat. 

\begin{example}\label{ex:BHEcpt}
We set 
$$
m=1,\qquad a=b=1,\qquad D=2,\qquad c=4,
$$
and accordingly 
$$
Z=Z_1+Z_2+4Z_3,\qquad U=Z_1+3Z_2+5Z_3.
$$
The vectors  $Z,Z_3,U$ are linearly independent and we may choose the central torus to be
$$
T^3=\mathfrak z/2\pi\mathrm{span}_{\mathbb Z}\{Z,Z_3,U\}.
$$
In particular, $Z_3$ and $U$ are primitive integral.
Let $A=\mathrm{diag}(i,-i)\in\mathfrak{su}(2)$ and define the compact connected subgroups 
$$
H=\{\exp(s(A+Z)):s\in\mathbb R/2\pi\mathbb Z\}, \qquad
K=H\exp(\mathbb RZ_3),\qquad
K'=H\exp(\mathbb RU).
$$
Moreover,
$$
H\cap\exp(\mathbb RZ_3) =H\cap\exp(\mathbb RU)=\{1\}.
$$
Thus the slice representations
$$
\nu\left(h\exp(\theta Z_3)\right)z=e^{i\theta}z,\qquad
\nu'\left(h\exp(\theta U)\right)z=e^{i\theta}z
$$
are well defined and both have kernel $H$.

With our choices, we have 
$$
\Delta=D+2b+1=5,\qquad
Z'=Z_3+\frac15Z_1+\frac35Z_2=\frac15U.
$$
Therefore $\lambda=5$, and the compatibility condition at the second endpoint is satisfied with
$$
m'=\lambda\frac{m(D-1)}{\Delta}=1.
$$
Our theorem applies to
$$
M=(G\times_K\mathbb D^2) \underset{G/H}{\cup}(G\times_{K'}\mathbb D^2).
$$
Note that the $G$-action on $M$ is almost effective, as the ineffective kernel is a proper subgroup of $H$, hence it is finite. 
The manifold $M$ can be seen to be diffeomorphic to $S^3\times S^3$. Indeed, the subgroup
$$
Q=\mathrm{SU}(2)\times\exp(\mathbb RZ_3)\times\exp(\mathbb RU)
$$
has the same orbits as $G$ and satisfies
$$
Q\cap H=\{1\},\qquad
Q\cap K=\exp(\mathbb RZ_3),\qquad
Q\cap K'=\exp(\mathbb RU).
$$
Its group diagram is therefore the product of the transitive $\mathrm{SU}(2)$-action on itself and the standard $T^2$-action on $S^3\subset\mathbb C^2$. 
It then follows that $M\cong S^3\times S^3$. 
\end{example}

The following result focuses on the complex structure of the example above and can be applied to many other cases considered in our Theorem \ref{Tcompact}.

\begin{proposition}
The complex structure $J$ on $S^3\times S^3$ constructed in Example \ref{ex:BHEcpt} is not compatible with any Bismut-flat metric. 
In particular, it is not biholomorphic to any Samelson structure on $S^3\times S^3$. 
\end{proposition} 
\begin{proof}
It was proved in \cite[Corollary 5.1]{Ba} that on a compact Hermitian manifold with finite fundamental group and admitting a Bismut-flat metric, 
any other BHE metric is necessarily Bismut-flat. 
In our case, where $M=S^3\times S^3$, this implies that the complex structure we constructed is not compatible with any Bismut-flat metric. 
As for the second claim, suppose $(M,J)$ is biholomorphic to a $(M,J_{{\rm{Sam}}})$, for some Samelson complex structure. 
We can then pull-back our pluriclosed metric on $M$ and we can average it making it invariant under left translations, where $S^3\cong \SU(2)$. 
By a result of Fino and Grantcharov (see \cite[Theorem 3.2]{FG}), the complex structure would be compatible with a bi-invariant metric, 
hence the manifold would carry a Bismut-flat metric, a contradiction.
\end{proof}

\begin{theorem}\label{thm:BHEdeformation}
On the manifold $M=S^3\times S^3$ there exists a family of BHE structures $(g_\epsilon,J_\epsilon)$ depending smoothly on a real parameter 
$\epsilon \in [0,\delta)$, for some $\delta>0$, such that
\begin{enumerate}[$\bullet$]
\item for $\epsilon=0$, $(g_0,J_0)$ is Bismut-flat and $J_0$ is a Samelson complex structure;
\item for $0<\epsilon<\delta$, $(g_\epsilon,J_\epsilon)$ is BHE and not Bismut-flat. 
\end{enumerate}
\end{theorem}
\begin{proof} 
We fix $a_o\in\mathbb R$, $a_o\neq 0$, and put for $\epsilon\in \mathbb R^+$
$$
m=1,\quad c=\epsilon,\quad b_\epsilon=-1+\frac\epsilon 2,\quad D_\epsilon=a_o^2+b_\epsilon^2,
$$
so that $D_\epsilon>1$ if $\epsilon$ is a real positive number in a suitable right neighborhood $I=[0,\epsilon_o)$ of zero in $\mathbb R$. 
So, $\Delta_\epsilon = D_\epsilon+2b_\epsilon+1 = D_\epsilon-1+\epsilon>0$  for $\epsilon\in I$. 

Following Theorem \ref{Tcompact}, for $\epsilon\in I \smallsetminus \{0\}$ we consider the lattice $\Lambda_\epsilon$ in $\frak z$, 
where $\frak z$ is generated by $\{Z_1,Z_2,Z_3\}$, which is defined by 
$$
\Lambda_\epsilon = 2\pi \operatorname{span}_{\mathbb Z}\{Z_\epsilon,Z_3,\lambda_\epsilon Z_\epsilon'\},
$$
where we recall that 
$$
Z_\epsilon = {a_o}Z_1+b_\epsilon Z_2+ \epsilon Z_3,\qquad 
Z_\epsilon'= Z_3+\frac{a_o}{\Delta_\epsilon}Z_1+\frac{D_\epsilon+b_\epsilon}{\Delta_\epsilon}Z_2
$$ 
and $\lambda_\epsilon$ is a factor ensuring that the slice weight at the second orbit becomes $m'=1$. We know that 
$$
m'=\lambda_\epsilon\,\frac{m(D_\epsilon-1)}{\Delta_\epsilon}, 
$$
and if we put $m'=1$, we get 
$$
\lambda_\epsilon = \frac{\Delta_\epsilon}{D_\epsilon-1}.
$$
We now apply our Theorem \ref{Tcompact} and we obtain a manifold $M_\epsilon$ together with a Hermitian structure $(g_\epsilon,J_\epsilon)$,  
where the metrics $g_\epsilon$ are BHE, and have non-zero Bismut curvature for $\epsilon \in I\smallsetminus\{0\}$. 
Note that the acting group here is $G_\epsilon \coloneqq  T_\epsilon \times \SU(2)$, where $T_\epsilon \coloneqq \frak z/\Lambda_\epsilon$. 
We now consider a fixed torus given by $T\coloneqq \mathbb R^3/\Lambda$, where $\Lambda$ is the lattice generated by element 
$2\pi e_i$, $i=1,2,3$. We note that the vectors $\{Z_\epsilon,Z_3,Z_\epsilon'\}$ are linearly independent for $\epsilon\in I$, 
so that we can fix an isomorphism $T\to T_\epsilon$ given by the correspondence 
$$
e_1\mapsto Z_\epsilon,\quad e_2 \mapsto Z_3,\quad e_3\mapsto \lambda_\epsilon Z_\epsilon',
$$
and we can consider the fixed action of $G \coloneqq  T\times \SU(2)$ on $M_\epsilon$ obtained by composition with the isomorphism $T\cong T_\epsilon$. 
Now, the isotropy subgroups and their slice representations become fixed, namely 
$$
H = \exp(\mathbb R(A+e_1)),\qquad 
K = H\exp(\mathbb R e_2),\qquad 
K' = H\exp(\mathbb R e_3),
$$
with slice weights $m=m'=1$. 
This group diagram $(G,H,K,K')$ identifies $M_\epsilon$ diffeomorphically with $M \coloneqq S^3\times S^3$ (see e.g. \cite{Ho}). 
As the distance of the two singular orbits in $M_\epsilon$ depends on $\epsilon$, being given by $\frac \pi 2\sqrt{\frac {D_\epsilon}{D_\epsilon-1}}$, we choose an equivariant identification $\Phi_\epsilon : S^3\times S^3\to M_\epsilon$ so that 
$$
\Phi_\epsilon(u,gH) = g\cdot\gamma_\epsilon\left(\frac 1{\kappa_\epsilon}u\right),~\quad u\in(0,\pi),\qquad 
\kappa_\epsilon \coloneqq 2\sqrt{\frac{D_\epsilon-1}{D_\epsilon}}
$$
on the regular part $(0,\pi)\times G/H$. 
These diffeomorphisms extend smoothly across the singular orbits. Indeed, in a $G$-invariant tubular neighborhood of the first orbit with $u$ as a slice radius, 
the map $\Phi_\epsilon$ reads $[(g,u)]\mapsto [(g,u/\kappa_\epsilon)]$, which is smooth, and the same argument applies to the second singular orbit. 

We can then pull-back the Hermitian structures $(g_\epsilon,J_\epsilon)$ to $S^3\times S^3$. 
We keep the same notation and denote by $\o_\epsilon$ the corresponding fundamental forms. 
On the regular part with the coordinate $u=\kappa_\epsilon t\in (0,\pi)$, we obtain the following expression for the metrics 
\[
\begin{split}
g_\epsilon &= \frac 1{\kappa_\epsilon^2} du^2 + (\theta_{1,\epsilon} + p_\epsilon\theta_{3,\epsilon})^2+ 
(\theta_{2,\epsilon} + q_\epsilon\theta_{3,\epsilon})^2+\ell_\epsilon^2\theta_{3,\epsilon}^2 + D_\epsilon(x^2+y^2),\\
\o_\epsilon &= (\theta_{1,\epsilon} + p_\epsilon\theta_{3,\epsilon})\wedge (\theta_{2,\epsilon} + q_\epsilon\theta_{3,\epsilon}) 
+ \frac{\ell_\epsilon}{\kappa_\epsilon} du\wedge \theta_{3,\epsilon} + D_\epsilon \Omega, 
\end{split}
\]
where the functions are in the variable $u$ and the forms $\theta_{i,\epsilon}$ are smooth in $\epsilon$ after the identification of the central algebras. 
In order to give the explicit expressions of the involved functions, we set
$$
h_\epsilon \coloneqq \frac{1-\cos u}{2(D_\epsilon-1)},\qquad 
\mu_\epsilon \coloneqq 1 + \epsilon h_\epsilon , 
$$
and we obtain 
\beq\label{solepsilon}
p_\epsilon = -a_o \frac{h_\epsilon}{\mu_\epsilon},\qquad
q_\epsilon = -\frac{(D_\epsilon+b_\epsilon)h_\epsilon}{\mu_\epsilon},\qquad
\ell_\epsilon = \frac{\sin(u)}{\kappa_\epsilon \mu_\epsilon}.
\eeq
These functions depend smoothly on the pair $(u,\epsilon)$ near $[0,\pi]\times\{0\}$ and at $\epsilon=0$ they give 
$$
p_0=-\frac{1-\cos u}{2a_o},\qquad q_0= -\frac{1-\cos u}2,\qquad \ell_0= \frac{\sin u}{\kappa_0},
$$
which are precisely the solutions for $c=0$. 

We now need to verify that we have a smooth dependence on $\epsilon$ also in an invariant tubular neighborhood of a singular orbit. 
We will keep the same notation as in Section \ref{sub:extend}, where we dealt with the extendability conditions. 
We consider a tubular neighborhood $G\times_K \Upsilon$  of the first singular orbit. The same arguments can be applied also to the second singular orbit. 
In the slice $\Upsilon$ we adopt the same coordinates $(t,z)$  with $r^2=t^2+z^2$. Here $r$ agrees with the fixed radial coordinate $u$, so that the radial curve is represented by $u\mapsto(u,0)$. Now, along $\Upsilon$ we use the frame 
$$
\frac\partial{\partial t},~\frac\partial{\partial z},~\widehat Z_{1,\epsilon},~\widehat Z_{2,\epsilon},~\widehat X,~ \widehat Y,
$$
which depend smoothly on $\epsilon$. 
Since the slice weight is $m=1$, we have that the equivariant map representing $\theta_{3,\epsilon}$ is given by 
$$
\vartheta_{3,\epsilon} = \frac{tdz-zdt}{t^2+z^2}, 
$$
and the tensors $\theta_{1,\epsilon},\theta_{2,\epsilon},\Omega$ and $x^2+y^2$ extend smoothly across the singular orbit, with smooth dependence on $\epsilon$, 
by the same argument used in Section \ref{sub:extend}. Using the explicit solutions \eqref{solepsilon}, we can write in analogy with Section \ref{sub:extend}
$$ 
p_\epsilon(r) = r^2P_\epsilon(r^2),\quad 
q_\epsilon(r)= r^2 Q_\epsilon(r^2)\quad 
\ell_\epsilon(r)= rL_\epsilon(r^2),\quad 
L_\epsilon(0)=\frac 1{\kappa_\epsilon}, 
$$
where $P_\epsilon, Q_\epsilon, L_\epsilon$ are smooth jointly in $(\epsilon,r^2)$ near the zero section. Then along $\Upsilon$ we have 
$$
p_\epsilon\theta_{3,\epsilon} = P_\epsilon(t^2+z^2)(t dz - z dt),\qquad q_\epsilon\theta_{3,\epsilon} = Q_\epsilon(t^2+z^2)(t dz - z dt),
$$
showing that these expressions are globally smooth in $(t,z,\epsilon)$. 

We now have to deal with the radial term $\frac 1{\kappa_\epsilon^2}du^2+ \ell_\epsilon^2\theta_{3,\epsilon}^2$. We observe that 
$$
L_\epsilon(r^2)= \frac 1{\kappa_\epsilon} + r^2B_\epsilon(r^2),
$$
for some smooth $B_\epsilon(r^2)$. As $dr = \frac 1r(tdt+zdz)$ we have 
\[
\begin{split}
\frac 1{\kappa_\epsilon^2} dr^2 + \ell_\epsilon^2\theta_{3,\epsilon}^2 
& = \frac 1{\kappa_\epsilon^2} \frac{(tdt+zdz)^2}{r^2} + r^2\left(\frac 1{\kappa_\epsilon^2} + r^4B_\epsilon(r^2)^2 + 2 \frac{r^2}{\kappa_\epsilon}B_\epsilon(r^2)\right)\cdot \frac{(tdz-zdt)^2}{r^4}\\
 {}&= \frac 1{\kappa_\epsilon^2}\left(\frac{t^2dz^2+z^2dt^2}{r^2}+\frac{t^2dz^2+z^2dt^2}{r^2}\right) + \left(r^2B_\epsilon(r^2)^2+\frac 2{\kappa_\epsilon}B_\epsilon(r^2)\right)\cdot (tdz-zdt)^2 \\
 {}&= \frac 1{\kappa_\epsilon^2}\left(dt^2+dz^2\right) + \left(r^2B_\epsilon(r^2)^2+\frac 2{\kappa_\epsilon}B_\epsilon(r^2)\right)\cdot (tdz-zdt)^2,
\end{split}
\]
and this shows that this radial term is smooth in $\epsilon$. In a similar way we have 
$$
dr\wedge \theta_{3,\epsilon} = \frac 1r dt\wedge dz,
$$
and therefore
$$ 
\frac 1{\kappa_\epsilon} \ell_\epsilon(r) dr \wedge \theta_{3,\epsilon} = \frac 1{\kappa_\epsilon} L_\epsilon(t^2+z^2) dt\wedge dz,
$$
again showing the claimed smoothness. 
Summing up, we have proved that $g_\epsilon,\o_\epsilon$ extend smoothly over the singular orbits with smooth dependence on the parameter $\epsilon$. Consequently, $g_\epsilon$ and $\omega_\epsilon$ depend smoothly on $\epsilon$ on the whole fixed manifold $M$. 
Since $g_0$ is positive definite, the identity $\omega_\epsilon(\cdot,\cdot)=g_\epsilon(J_\epsilon\cdot,\cdot)$ also implies smooth dependence of $J_\epsilon$.
This means that $(g_\epsilon,J_\epsilon)$ converge to $(g_0,J_0)$ smoothly. 
Then $(M, g_0,J_0)$ is Bismut flat by Lemma \ref{c=0} and therefore $J_0$ is a Samelson structure by \cite{WYZ}.
\end{proof}

\begin{remark}
We do not know whether the deformation of the complex structure considered in this theorem is the one prescribed by \cite[Theorem A]{Ba2}  
for pluriclosed structures. 
\end{remark}

We conclude this section giving some general information on the topology of the compact manifold $M$ of Theorem \ref{Tcompact}, 
at least in the generic case $a\neq 0$.
\begin{proposition}\label{prop:topology}
The manifold $M$ in Theorem \ref{Tcompact} is finitely covered by $S^3\times S^3$ if $a\neq 0$.
\end{proposition}

\begin{proof} 
We first prove that there is a compact subgroup $L$ of $G$ with the same orbits as $G$. 

We consider the two circles generated by $Z_3,Z'$ and the $2$-torus $S$ generated by them. 
We claim that $L:=\SU(2)S$ has the same orbits as $G$. We observe that the vector $Z$ does not belong to the span of $Z_3,Z'$. 
Indeed, if we construct the matrix of coefficients of $Z,Z_3,Z'$ with respect to the basis $\{Z_1,Z_2,Z_3\}$ and compute its determinant, we obtain 
$$
\det \left(\begin{matrix} a&b&c\\ 0&0&1\\ \frac{ma}\Delta& \frac{m(D+b)}\Delta& 1\end{matrix}\right) = -\frac {amD}\Delta \neq 0\qquad (a\neq 0).
$$
It then follows that 
$$
\frak l \cap \frak h = \{0\},\qquad \frak g = \frak h + \frak l,
$$
so that $L$ acts transitively on the principal orbits and our claim is proved. 
We denote by $H_L = H \cap L$ the principal isotropy and by $K_L=K\cap L,K'_L=K'\cap L$ the two singular isotropies. We are therefore reduced to the $L$-action.

Now, we consider the subgroup $N\coloneqq SH_L$, which contains the two singular isotropy subgroups of $L$, as $K_L = K_L^oH_L$ and $K_L' = (K_L')^oH_L$.  
Therefore the $L$-action is not primitive and the manifold admits a description as a fibre bundle (see e.g. \cite{AA})
$$
M = L\times_N F,
$$
where the fibre $F$ is a $3$-manifold which is of cohomogeneity one under the $N$-action. 

The base $L/N$ can be seen as a quotient of $\SU(2)$ and therefore is covered by $S^3$. 
We can now pull back the bundle $SU(2)\to L/N$ along the fibering $M\to L/N$ obtaining the following diagram 
$$
\begin{tikzcd}[column sep=large, row sep=large]
M' \arrow[r] \arrow[d]
& SU(2)\cong S^3 \arrow[d, "q"] \\
M=L\times_N F \arrow[r, "\pi"]
& L/N
\end{tikzcd}
$$
where $M'=\{(g,m)\in \SU(2)\times M|\ q(g)=\pi(m)\}.$ Then the pulled back bundle is trivial, as the map 
$$
\Phi: S^3\times F\to M',\quad \Phi((g,f)) = (g,[(g,1),f])
$$
is a bundle diffeomorphism. It then follows that $M$ is finitely covered by $M'\cong \SU(2)\times F$.

Now the $3$-dimensional manifold $F$ is acted on by $N$ and since $H_L$ is a finite subgroup, the action of the $2$-torus $S$ has the same orbits in $F$ as $N$.
Moreover, the two singular $S$-orbits in $F$ are circles corresponding to two different directions in the Lie algebra $\frak s$ of $S$. 
By Van Kampen's Theorem, we have $\pi_1(F) \cong \mathbb Z^2/C$, where $C$ is a rank two subgroup, hence $\pi_1(F)$ is finite. 
By Perelman's resolution of the Poincar\'e conjecture, $F$ is finitely covered by $S^3$. Therefore, $M$ is finitely covered by $M'$, 
which is finitely covered by $S^3\times S^3$, and our claim is proved.\end{proof}

\subsection{The restricted Bismut holonomy}\label{sec:hol} 
We now compute the restricted holonomy group of the Bismut connection $\nabla^B$ of our examples. 
We recall that the condition $\rho^B=0$ implies that the holonomy algebra $\mathfrak{hol}(\nabla^B)$ is contained in $\mathfrak{su}(3)$. 
In all our examples, the fields $V$ and $JV$ are Bismut parallel, therefore $\mathfrak{hol}(\nabla^B)\subseteq\mathfrak{su}(2)$. 
Note that this is always the case for compact non-K\"ahler BHE threefolds by \cite[Prop.~2.6]{ABLS}. 
We now prove the following. 
\begin{proposition}\label{prop:hol}
    Let $(M,g,J)$ be a compact non-K\"ahler BHE threefold with non-flat Bismut connection admitting an almost effective cohomogeneity one action 
    of $G = T^3\times \SU(2)$ by holomorphic isometries. Then, 
    \[
    \mathrm{Hol}^o(\nabla^B) \cong T^1 \subset \SU(3). 
    \]
\end{proposition}
\begin{proof}
    Using the setting introduced in the previous sections, we consider the normal geodesic $\gamma:[0,T]\to M$ with $\gamma(0)$ and $\gamma(T)$ 
    singular points. For every $t\in[0,T]$, we define the $G$-invariant distribution $\mathcal{W}$ that coincides with the span of
    the linearly independent vectors $\hat{X}_{\gamma(t)}$, $\hat{Y}_{\gamma(t)}$ at $\gamma(t)$. This distribution is well-defined because the space 
    generated by $\hat{X}_{\gamma(t)}$, $\hat{Y}_{\gamma(t)}$ is $G_{\gamma(t)}$-invariant. 
    Moreover, it is $J$-invariant since $J\hat{X}=\hat{Y}$ along $\gamma$, and it is orthogonal to the distribution $\mathcal{V}$ spanned by $V,JV$. 
    We proved that the $2$-form $\Omega$ smoothly extends over the two singular orbit and therefore $\mathcal W = (\ker \Omega)^\perp$ is a smooth distribution everywhere. 
    
    We claim that $\mathcal{W}$ is $\nabla^B$-parallel. By continuity, it is sufficient to show this on the regular part $M_{\mathrm{reg}}$ of $M$, 
    which  we identify with $(0,T)\times G/H$; here $\g(t)$ corresponds to the curve $t\mapsto (t,eH)$. 
    We fix a local section $\sigma : \mathcal{U}\subseteq G/H \to G$ over an open neighborhood of $u_o = eH$ with $\sigma(u_o)=e$, so that 
    \[
    \sigma(u)\cdot\gamma(t) = (t,u).
    \]
    We then define two sections $E_1$ and $E_2$ of $\mathcal{W}|_{(0,T)\times\mathcal{U}}$ as follows
    \[
    E_1|_{(t,u)} = d\sigma(u)|_{(t,u_o)}\hat{X}_{(t,u_o)}, \qquad E_2|_{(t,u)} = d\sigma(u)|_{(t,u_o)}\hat{Y}_{(t,u_o)}.
    \]
    Then, $\{E_1,E_2\}$ is a local frame of $\mathcal{W}$ and, since $G$ acts by holomorphic isometries,  
    \[
    \| E_1\|^2 = \phi = D, \qquad E_2|_{(t,u)} = JE_1|_{(t,u)}.
    \]

    Now, consider two vectors $w_1,w_2\in \mathcal{W}_{\g(t)}^\perp$ and the map 
    $$
    \psi: \mathcal W_{\g(t)}\ni x \mapsto g\left(\n^B_{w_1}\tilde x,w_2\right),
    $$
    where $\tilde x$ is an arbitrary extension of $x$ as a section of $\mathcal{W}$. The map $\psi$ is well-defined and it gives a $G_{\g(t)}$-invariant element 
    in $(\mathcal W_{\g(t)})^*$, hence it vanishes identically. Therefore, we only need to prove that 
    $$
    \n^B_{E_1}E_1 \in \Gamma(\mathcal W),\qquad \n^B_{E_2}E_2\in \Gamma(\mathcal W),
    $$
    as $\n^B_{E_1}E_2 = J\n^B_{E_1}E_1$ and $\n^B_{E_2}E_1 = -J\n^B_{E_2}E_2$.
    Now we observe that the torsion of $\n^B$ is skew, so that $\n^B_{E_1}E_1 = \nabla^{LC}_{E_1}E_1$. We claim that for every $w\in \mathcal W_{\g(t)}^\perp$
    \[
    g\left(\nabla^{LC}_{E_1}E_1,w\right)|_{\g(t)} =  0.
    \]
    From this we will have $\n^B_{E_1}E_1 \in \Gamma(\mathcal W)$ (and similarly for $E_2$).
    
    First of all, we know that $g(E_1,\xi)=0$ by $G$-invariance, hence 
    $$ 
    g\left(\nabla^{LC}_{E_1}E_1,\xi\right)|_{\g(t)} = - g\left(E_1,\nabla^{LC}_{E_1}\xi\right)|_{\g(t)}.
    $$
    Moreover, by the construction of $E_1$ and the fact that $\sigma$ is independent of $t$, we have $[E_1,\xi]=0$ and therefore 
    $$
    g\left(\nabla^{LC}_{E_1}E_1,\xi\right)|_{\g(t)} = - g\left(E_1,\nabla^{LC}_{E_1}\xi\right)|_{\g(t)} = -\frac 12 \xi g(E_1,E_1) = -\frac 12 \phi' = 0.
    $$
    We observe that $\mathcal W_{\g(t)}^\perp$ is spanned by $\xi,J\xi,V,JV$ at $\g(t)$ and, moreover, for each fixed $t$, $J\xi_{\g(t)}=\hat U$ for some 
    non-zero central $U\in\frak z$. It then follows that it is enough to prove 
    $$ 
    g\left(\nabla^{LC}_{E_1}E_1,\hat F\right)|_{\g(t)} =  0,  
    $$ 
    for every $F\in\frak z$. Now, at $\g(t)$ we have 
    $$
    g(\hat{\frak z}_{\g(t)},\mathcal W_{\g(t)}) = 0,
    $$
    and since both $\hat{\frak z}$ and $\mathcal W$ are $G$-invariant, we have that $g(E_1,\hat F) = 0$ everywhere. Hence, 
    $$
    g\left(\nabla^{LC}_{E_1}E_1,\hat F\right)|_{\g(t)} = - g\left(\nabla^{LC}_{E_1}\hat F,E_1\right)|_{\g(t)} =0, 
    $$
    as $\hat F$ is Killing. This proves our claim and shows that  $\mathcal{W}$ is preserved by $\nabla^B$. 
    Summing up, the Bismut holonomy acts trivially on the space spanned by $V,JV$ and preserves the rank $2$ distribution $\mathcal W$, 
    hence $\dim \mathrm{Hol}^o(\nabla^B) \leq 1$. As the Bismut connection is not flat, by the Ambrose-Singer Theorem the connected holonomy is not trivial 
    and we get  $\mathrm{Hol}^o(\nabla^B) \cong T^1 \subset \SU(3)$.   
\end{proof}

\begin{remark}
The same proof works for the non-compact complete examples of Theorem \ref{thm:noncptex}, 
therefore $\mathrm{Hol}^o(\nabla^B) \cong T^1 \subset \SU(3)$ also in that case. Indeed, the main ingredient is the constancy of the function $\phi$. 
This condition has been proven to hold in the compact case, while it is an Ansatz in the construction of the non-compact examples.
\end{remark}

\subsection{Transverse K\"ahler geometry}\label{sec:regularity}
We conclude the discussion focusing on the transverse K\"ahler geometry of the cohomogeneity one BHE threefolds determined in Theorem \ref{Tcompact}.

\smallskip

Recall that a compact non-K\"ahler BHE threefold $(M,g,J)$ admits a rank 2 Riemannian foliation $\mathcal{V} = \mathrm{span}(V,JV)\subset TM$ 
that is endowed with an invariant transversal K\"ahler structure $(g^T_K,J^T,\o^T_K)$ \cite[Theorem 1.1]{ABLS}. 
According to \cite[Definition 2.2]{ALL}, $(M,g,J)$ is said to be {\em quasi-regular} if $V,JV$ generate a $2$-dimensional torus $T^2_{\mathcal{V}}$ of 
biholomorphic isometries. Since $V$ and $JV$ are everywhere linearly independent, this torus acts locally freely on $M$, and the quotient 
$M_T\coloneqq M/T^2_{\mathcal{V}}$ is a compact orbifold carrying the induced K\"ahler orbifold structure $(g^T_K,J^T,\o^T_K)$. 
If the torus action is free, so that $M_T$ is smooth, then $(M,g,J)$ is said to be {\em regular}.  

In the cohomogeneity one setting considered here, it follows from \eqref{metric} that the transversal K\"ahler metric on the regular part 
$M_{\mathrm{reg}} = (0,T) \times G/H$ of $M$ is given by 
\[
g^T_K = \mu\left( dt^2 + \ell^2\,\theta_3^2 + D (x^2+y^2)\right).
\]
Since $\mu$ and $\ell$ are basic functions, and $dt$, $\theta_3$ and $x^2+y^2$ are basic tensors with respect to the foliation $\mathcal{V}$, the expression descends to local transverse spaces. 
When $c\neq0$, the condition $\mu' = \frac{2c}{D}\mu\ell$ allows us to use $\mu$ as a coordinate on $(0,T)$, yielding 
\[
g^T_K = \frac{\mu}{(\mu')^2} d\mu^2 +\left(\frac{D}{2c}\right)^2\frac{(\mu')^2}{\mu}\theta_3^2+ D\mu(x^2+y^2),
\]
and 
\[
\o^T_K = \frac{D}{2c}d\mu\W\theta_3 + D\mu\, x\W y, 
\]
which is easily seen to be K\"ahler thanks to the identity $d\theta_3 = 2c\,x\W y.$

\smallskip

\begin{example}\label{ex:quasireg}
The compact non-K\"ahler BHE threefold $(M,g_{1,1,1},J_{1,1,1})$ constructed in Example \ref{ex:BHEcpt} is quasi-regular. Indeed, the integral vectors
\[
Z-4Z_3=Z_1+Z_2,\qquad
U-5Z_3=Z_1+3Z_2
\]
are a basis of $\mathfrak{v}=\mathrm{span}_{\mathbb R}\{Z_1,Z_2\}$. Thus, $\mathfrak{v}$ exponentiates to a closed two-dimensional torus $T^2_{\mathcal{V}}$ of holomorphic isometries. 

\end{example}

\smallskip

The next general result proves that, in particular, none of the compact cohomogeneity one non-K\"ahler BHE threefolds of Theorem \ref{Tcompact} is regular. 

\begin{theorem}
Let $(M,g,J)$ be a compact connected BHE threefold admitting an almost effective cohomogeneity one action by holomorphic isometries of $G=T^3\times \SU(2)$. 
If $(M,g,J)$ is regular, then the Bismut connection is flat.
\end{theorem}
\begin{proof} 
By our assumption, $M_T$ is a compact K\"ahler complex smooth surface, which is acted on by a group locally isomorphic to 
$L \coloneqq T^1\times \SU(2)$ with cohomogeneity one. By contradiction, we suppose that $M$ is not Bismut-flat. 
Then, $M$ has two singular orbits of Type II, so that the surface $M_T$ has two singular $L$-orbits which are two homogeneous complex curves, 
hence copies of $\mathbb P^1$. 
In this case, $M_T$ is a union of two disk bundles over $\mathbb P^1$ with connected intersection, and by Van Kampen's Theorem we have 
$\pi_1(M_T)=\{1\}$, 
hence $b_1(M_T)=0$. By a result due to Oeljeklaus \cite{Oe2}, the surface $M_T$ is projective algebraic.  
The complexified group $L^\mathbb C$ acts on $M_T$ in an almost homogeneous fashion with a disconnected exceptional set. 
By the classification of such manifolds (see e.g. \cite[Prop.~3.1]{HS}), there exists an algebraic $L$-equivariant modification $\tilde M_T\to M_T$ 
and an $L$-homogeneous $\mathbb P^1$-bundle $\tilde M_T\to Q$ onto an $L$-flag manifold $Q$, where $Q\cong \mathbb P^1$. 
As the exceptional set of $M_T$ already consists of complex codimension $1$ curves, the modification is trivial and $\tilde M_T = M_T$. 
This shows that $M_T$ is biholomorphic to a homogeneous $\mathbb P^1$-bundle over $\mathbb P^1$, i.e., a Hirzebruch surface $F_k$. 
By Case 1 in the proof of \cite[Theorem 1.4]{ALL}, $k=0$ and the manifold $M$ is Bismut-flat, a contradiction.
\end{proof}

\begin{remark} 
Generically the compact cohomogeneity one non-K\"ahler BHE threefolds of Theorem \ref{Tcompact} are not quasi-regular, 
but for suitable choices of the parameter they become quasi-regular, like in Example \ref{ex:BHEcpt}. 
In  the quasi-regular case, we can quotient out the ineffectivity kernel of the $G$-action and we obtain that the orbifold singular locus is contained 
only in the singular orbits for the residual cohomogeneity one action on the quotient.  
\end{remark}

\section{Comparison with other compact BHE threefolds}\label{{sec:comparison}}
\label{sec:comparison}

In this section, we compare the compact BHE structures constructed above with the explicit examples of Apostolov-Lahdili-Lee \cite{ALL} and Apostolov-Lahdili-Lee-Rochon \cite{ALLR}. 
The comparison with the orthotoric examples is based on an obstruction to cohomogeneity one expressed in terms of transverse curvature invariants. 
The examples obtained by gluing are distinguished by their topology.

\subsection{The orthotoric examples}
Let $(M,g,J)$ be a compact connected non-K\"ahler BHE threefold, and consider the quotient $M_T = M/T^2_\mathcal{V}$ endowed with the transverse K\"ahler structure $(g^T_K,J^T,\o^T_K)$. 
We regard the transverse scalar curvature and the squared norm of its differential,
\[
S\coloneqq\operatorname{Scal}_{g^T_K},\qquad
\Sigma\coloneqq|dS|_{g^T_K}^2,
\]
as basic functions on $M$.

We now show that $dS\wedge d\Sigma=0$ on every compact cohomogeneity one non-K\"ahler BHE manifold. 
\begin{lemma}\label{lem:C1transverse}
Suppose that $(M,g,J)$ admits a cohomogeneity one action by holomorphic isometries. Then
\[
dS\wedge d\Sigma=0.
\]
\end{lemma}

\begin{proof}
Every holomorphic isometry preserves the normalized soliton potential and the vector fields $V$, $JV$. 
It therefore preserves the foliation $\mathcal V$ and the transverse K\"ahler metric $g^T_K$, and hence the functions $S$ and $\Sigma$. 
On the regular part of a cohomogeneity one manifold, both functions depend only on the parameter of the normal geodesic. 
Thus $dS\wedge d\Sigma=0$ there, and the identity holds everywhere by continuity.
\end{proof}

\begin{proposition}\label{prop:comparison1}
The explicit orthotoric BHE structures constructed in \cite[Section~9]{ALL} admit no cohomogeneity one action by holomorphic isometries. 
In particular, none of the compact BHE structures constructed in this paper is holomorphically homothetic to one of those structures.
\end{proposition}

\begin{proof}
On an open set, the transverse K\"ahler metric of the examples in \cite[Section~9]{ALL} has the orthotoric expression
\[
g^T_K = (x-y)\left(\frac{dx^2}{A(x)} +\frac{dy^2}{B(y)}\right) +\frac{1}{x-y}\left(A(x)(d\tau_1+y\,d\tau_2)^2+B(y)(d\tau_1+x\,d\tau_2)^2\right),
\]
where $x>y$ are the orthotoric coordinates defined on two real intervals, $\tau_1,\tau_2$ are angular coordinates on $T^2$, and 
$A$ and $B$ are quadratic positive functions on their respective coordinate intervals.
By \cite[(9.9)]{ALL}, the transverse scalar curvature is 
\[
S = \mathrm{Scal}_{g^T_K}=\frac{C}{x-y}>0,
\]
where $C = -A''-B''$ is a positive constant. 

Since
\[
(g^T_K)^{-1}(dx,dx)=\frac{A(x)}{x-y},\qquad
(g^T_K)^{-1}(dy,dy)=\frac{B(y)}{x-y},\qquad
(g^T_K)^{-1}(dx,dy)=0,
\]
we obtain
\[
\Sigma=\frac{C^2(A(x)+B(y))}{(x-y)^5}.
\]
Consequently,
\[
dS\wedge d\Sigma =-\frac{C^3(A'(x)+B'(y))}{(x-y)^7}\,dx\wedge dy.
\]
This expression cannot vanish identically. Indeed, the identity $A'(x)+B'(y)=0$ on an open coordinate rectangle would imply $A''=B''=0$, 
and hence $\mathrm{Scal}_{g^T_K}=0$, a contradiction. The first assertion follows then from Lemma \ref{lem:C1transverse}.

A holomorphic homothety would transfer the cohomogeneity one action of one of our examples to an action by holomorphic isometries on the corresponding orthotoric
example. This proves the second assertion.
\end{proof}

\subsection{The examples obtained by gluing}

The compact examples constructed in \cite[Theorem~3]{ALLR} are defined on the manifolds
\[
(\#3(S^2\times S^4))\,\#\,(\#4(S^3\times S^3)),\qquad
(\#5(S^2\times S^4))\,\#\,(\#6(S^3\times S^3)),\qquad
S^1\times\left(\#9(S^2\times S^3)\right).
\]
It is easy to see that the second Betti number of these manifolds is at least $3$.
For our compact examples, the orbit structure gives the following bound.

\begin{proposition}\label{prop:comparison2}
Every compact manifold obtained by the construction of Theorem \ref{Tcompact} has second Betti number at most two. 
Consequently, none is diffeomorphic to a manifold appearing in \cite[Theorem~3]{ALLR}.
\end{proposition}

\begin{proof}
Let $P=G/H$ be a principal orbit and let
$G/K,G/K'$ be the two singular orbits. As established in Lemma \ref{lem:b1GH}, $b_1(P)=2$. 
Moreover, each singular orbit is finitely covered by $S^3\times S^1$, see Remark \ref{rem:singorb}. 
Since pull-back under a finite covering is injective in real cohomology, it follows that both singular orbits have vanishing second Betti number. 

The manifold $M$ is the union of tubular neighborhoods of $G/K$ and $G/K'$, whose intersection deformation retracts onto $P$. 
The Mayer-Vietoris sequence therefore contains
\[
H^1(P,\mathbb R)
\rightarrow H^2(M,\mathbb R)
\rightarrow
H^2(G/K,\mathbb R)\oplus H^2(G/K',\mathbb R)=0.
\]
Thus $b_2(M)\leq b_1(P)=2$, which proves both assertions.
\end{proof}

\smallskip

\noindent{\bf Acknowledgements.} 
The authors are grateful to Vestislav Apostolov, Lorenzo Foscolo, Mario Garcia-Fern\'andez and Jeffrey Streets for stimulating and useful conversations which 
led to an improvement of the results contained in a first draft of this work.

During the preparation of this article, the authors used chatGPT to assist in reviewing the relevant literature and double checking their computations. 

The authors acknowledge partial support by GNSAGA (INdAM, Italy).

\end{document}